\documentclass{article}

\usepackage[margin = 1.6in]{geometry}

\usepackage{tikz}
\usepackage{tikz-cd}
\usepackage{url, hyperref}
\usepackage{float}

\tikzset{->-/.style={decoration={
  markings,
  mark=at position .45 with {\arrow{>}}},postaction={decorate}}}
\usepackage{amsmath}
\usepackage{amssymb}
\usepackage{enumitem}
\usepackage{graphicx}
\usepackage{mathdots}
\usepackage{color}
\usepackage{diagbox}
\usepackage{array, makecell}
\usepackage{rotating}
\usepackage{amsthm}
\usepackage{adjustbox}
\usetikzlibrary{arrows}

\newtheorem{definition}{Definition}
\newtheorem{theorem}[definition]{Theorem}

\newtheorem{proposition}[definition]{Proposition}
\newtheorem{corollary}[definition]{Corollary}
\newtheorem{lemma}[definition]{Lemma}
\newtheorem{remark}[definition]{Remark}
\newtheorem{notation}[definition]{Notation}

\def\O{\mathcal{O}}

\def\Z{\mathbb{Z}}
\def\Pp{\mathbb{P}\,}

\def\R{\mathbb{R}}

\def\P{\mathbb{P}^1}

\def\R{\mathcal{R}}
\def\CH{\mathrm{CH}}
\def\Sym{\mathrm{Sym}}

\title{The integral Chow ring of $\R_2$}
\author{Alessio Cela and Aitor Iribar L\'opez}
\date{\vspace{-5ex}}

\begin{document}

\maketitle

\begin{abstract}
In this paper we compute the integral Chow ring of the moduli stack $\R_2$ of Prym pairs of genus 2.

\end{abstract}

MSC 2020: 14C15 (primary), 14H10

\tableofcontents

\section{Introduction}

Prym curves have been studied for more than 100 years due to their connections to the theory of curves, admissible covers and abelian varieties \cite{M,schottky, H. Farkas, Bea, Beabis, DonagiI, DonagiII}. They are connected, \'etale double covers of a smooth, projective curve or, equivalently, non-trivial line bundles on the base curve that are 2-torsion in the Picard group.

After fixing the genus $g$ of the base curve, the moduli stack of Prym curves has a coarse moduli space, which is denoted by $R_g$, and it is a finite cover of the moduli space of genus $g$ smooth curves $M_g$, of degree $2^{2g}-1$. Just as for $M_g$, it is natural to study the geometric properties of $R_g$.

The study of the Chow rings of $M_g$ started with Mumford \cite{Mumford}. With rational coefficients, they are known for genus up to $9$ \cite{Faber, Izadi, PenevVakil, SamHannah}, but the torsion elements in the Chow ring of the moduli stack $\mathcal{M}_g$ have drawn attention only recently. The full Chow group of $\mathcal{M}_2$ was computed in \cite{Vistoli}, relying on the fact that the curves in question are hyperelliptic. Following this direction, the Chow groups of the moduli stack of (pointed) hyperelliptic curves of any genus have been studied in \cite{EF, Pernice2, Lorenzo}.

The integral Chow groups of the Deligne-Mumford compactification $\overline{\mathcal{M}}_2$ were independently computed in \cite{Eric, VistoliLorenzo} and the ones of $\overline{\mathcal M}_3$ after inverting $6$ by \cite{Pernice}.

On the other hand, there are fewer results for the moduli stack of Prym curves, although the Picard groups with rational coefficients can be obtained from the results in \cite{Putman}. This paper aims to begin filling this gap by computing the integral Chow ring of the moduli stack of genus $2$ Prym pairs. It also initiates a series of works computing the integral Chow rings of hyperelliptic Prym pairs \cite{CelaLandiI, CelaLandiII}, with another paper forthcoming.

We should remark that, even though there is no ambiguity on what the stack $\mathcal M_g$ should parametrize, there are two natural choices of stacks that have $R_g$ as a coarse moduli space (see \S\ref{sec: background facts about Rg} below). We will denote them by $\widetilde{\R}_g$ and $\R_g$, and the second one is the rigidification (in the sense of \cite{ACV}) of the first one along the generic $\mathbb Z/ 2 \mathbb{Z}$ automorphism coming from the \'etale double cover. This distinction had already appeared in \cite{Cor2, BCF}.

\subsection{Results and methods}

In this paper, we compute the integral Chow ring of the stack $\R_2$ (defined over an algebraically closed field $k$ of characteristic different from $2$ and $3$), using the techniques of equivariant intersection theory developed in \cite{EG}. The first step is to find an explicit presentation of $\R_2$.

To state our results, we require the following definition.

\begin{definition}\label{notation: representations}
    Let
    $$
    G= (\mathbb{G}_m \times \mathbb{G}_m) \rtimes \Z/ 2 \Z
    $$
    where the action of $\Z/ 2\Z$ on $\mathbb{G}_m \times \mathbb{G}_m$ permutes the factors.

    We denote by $\Gamma$ the representation of $G$ arising from the sign representation of the
    $\Z/2\Z$ quotient of $G$. 
    
    Furthermore, we use $V$ to denote the standard representation of $G$ coming from the inclusion
    $
    G  \subseteq \mathrm{GL}_2
    $ 
    as the subgroup preserving the set of lines $\{ k (1,0), k(0,1)\} \subseteq k^2$. 
\end{definition}

\begin{theorem}\label{thm: presentation R2}
    We have an isomorphism of algebraic stacks
    $$
    \bigg[ \frac{\mathrm{Sym}^4 (V^\vee) \otimes \mathrm{det}(V) \otimes \Gamma \smallsetminus \Delta}{G}\bigg] \xrightarrow{\sim} \R_2.
    $$
    where 
    %we view the numerator $\mathrm{Sym}^4 (V^\vee) \otimes \mathrm{det}(V) \otimes \Gamma $\Ait{I think this is clear} as a $G$-representation and 
    $\Delta$ is the locus of polynomials having either a root at $0$ or $\infty$, or having a double root.
\end{theorem}

The proof will be given in \S\ref{sec: proof of presentation}. Here, we content ourselves to explain the main idea. The starting point is \cite[Lemma 4.3]{Verra}, which we recall here in the case $g=2$.

Let $C$ be a smooth genus $2$ curve defined over an algebraically closed field $k$ of characteristic different from $2$\footnote{ In \cite{Verra} the field $k$ is assumed to be the field of complex numbers, but the result we are interested in remains true (with the same proof) over every algebraically closed field of characteristic different from $2$.} and let $q : C \to \mathbb{P}^1$ be its associated double cover (well-defined up to an isomorphism of $\mathbb{P}^1$). Then the ramification divisor of $q$ is the set
$$
W = \{ w_1, \ldots , w_{6} \}
$$
of the Weierstrass points of $C$. Let $H=q^* \mathcal{O}_{\mathbb{P}^1}(1)$ be the hyperelliptic line bundle of $C$ and $E$ be the family of effective divisors of cardinality $2$, which are supported on $W$. For $e \in E$ the line bundle $H \otimes \mathcal{O}_C(-e) $ is a non-trivial square root of $\mathcal{O}_C$, i.e. an element of $\mathrm{Pic}^0(C)[2] \smallsetminus \{ \mathcal{O}_C \}$. The following lemma is proven in \cite{Verra}.

\begin{lemma}\cite[Lemma 2.3]{Verra}\label{lemma: idea}
    The map 
    \begin{align*}
    &E \to \mathrm{Pic}^0(C)[2] \smallsetminus \{ \mathcal{O}_C \} \\
    &e \mapsto H \otimes \mathcal{O}_C(-e)  
    \end{align*}
    is a bijection.
\end{lemma}

Thus, for $\eta \in \mathrm{Pic}^0(C)[2] \smallsetminus \{ \mathcal{O}_C \}$, we can uniquely write $\eta= H \otimes \mathcal{O}_C(-e)$ for $e \in E$ and, up to composing $q$ with an isomorphism of $\mathbb{P}^1$, assume that $q(e)= \{ 0, \infty \} \subseteq \mathbb{P}^1$. Therefore, we have an isomorphism
$$
\eta^{\otimes 2} = q^* ( \mathcal{O}_{\mathbb{P}^1}(2) \otimes \mathcal{O}_{\mathbb{P}^1}(- [0] - [\infty])) \xrightarrow{\sim} \mathcal{O}_C
$$
given by the polynomial $XY \in H^0(\mathbb{P}^1,\mathcal{O}_{\mathbb{P}^1}(2) \otimes \mathcal{O}_{\mathbb{P}^1}(- [0] - [\infty]) ) $.

This provides a map 
\begin{equation}\label{eqn: G-torsor presentation}
 \mathrm{Sym}^4 (V^\vee) \smallsetminus \Delta \to \R_2
\end{equation}
which Theorem \ref{thm: presentation R2} states is a $G$-torsor. We refer to \eqref{eqn: def C}, \eqref{eqn: def eta}, \eqref{eqn: def trivialization}, \eqref{eqn: def map on morphism part 1} and Lemma \ref{lemma: def of tau} for the actual definition of the map \eqref{eqn: G-torsor presentation}.

Next, in \S\ref{sec: computation of Chow}, we use the machinery of equivariant intersection theory and the presentation of $\R_2$ in Theorem \ref{thm: presentation R2} to compute the Chow ring of $\R_2$. Finally, in \S\ref{sec: tautological classes}, we identify the generators as naturally defined geometric classes on $\R_2$. The final result is the following.

\begin{theorem}\label{thm: Chow of R2}
    Over any algebraically closed base field of characteristic distinct from $2$ and $3$, the Chow ring of
    the moduli space of Prym pairs of genus 2 is given by
    $$
    \CH^*(\R_2)=\frac{\mathbb Z[\lambda_1, \lambda_2, \gamma]}{(2\lambda_1, 2\gamma, 8\lambda_2, \gamma^2+\lambda_1\gamma, \lambda_1^2+\lambda_1 \gamma)}
    $$
    where $\lambda_1$ and $\lambda_2$ denote respectively the first and second Chern classes of the Hodge bundle, and $\gamma$ denotes the first class of the pushforward of the structure sheaf of a natural degree $2$ cover of $\R_2$ (see Lemma \ref{lemma: interpretation gamma}).
\end{theorem}

\textbf{Assumptions on Characteristic}: for the remainder of the paper, we work over an algebraically closed field $k$ of characteristic distinct from 2 and 3.

\subsection*{Acknowledgments}

Both authors participated in the AGNES 2023 Summer School on "Intersection theory on moduli spaces" at Brown
University, from where we learned most of the tools applied in this work. We thank the organizers Dan Abramovich, Melody Chan, Eric Larson, and Isabel Vogt the NSF grant DMS-2312088 for funding the conference.

We also would like to thank Younghan Bae, Samir Canning, Andrea di Lorenzo, Andrew Kresch, Alberto Landi, Sam Molcho, Michele Pernice and Johannes Schmitt for several discussions about generalities on stacks and their Chow rings, stacks of Jacobians and rigidifications. We are grateful to Gavril Farkas for conversations about the ambiguity in what the stack of Prym pairs can refer to, and to Rahul Pandharipande for inviting us to speak about this project at the ETH Algebraic Geometry and Moduli Seminars. Section \ref{sec: background facts about Rg} was added in response to his questions. Finally, we thank the referee for carefully reading the paper and for suggesting various improvements. A.C. was supported by SNF-219369 and SNF-222363,  A.I.L. received support from SNF-219369.

\section{Presentation of the stack of genus 2 Prym pairs}

\subsection{Definitions}

We start with recalling the definition of $\R_g$. 

\begin{definition}
    A Prym curve of genus $g$ is the datum of $(C, \eta, \beta)$ where $C$ is a smooth geometrically connected genus $g$ curve, $\eta \in \mathrm{Pic} (C)$ is non-trivial and $\beta: \eta^{\otimes 2} \to \mathcal{O}_C$ is an isomorphism of invertible sheaves on $C$.
\end{definition}

\begin{definition}
    A family of Prym curves of genus $g$ is a smooth family of genus $g$ curves $f : C \to S$ with an invertible sheaf $\eta$ on $C$ and a isomorphism $\beta: \eta^{\otimes 2} \to \mathcal{O}_C$
    such that the restriction of these data to any geometric fiber of $f$ gives rise to a Prym curve.
\end{definition}

\begin{definition} \label{defn: Rg}
    The prestack $\R_g^{\mathrm{pre}}$ is defined over the Big étale site $\mathrm{Sch}_{\text{ét}}$
    as the category whose objects over a scheme $S$ are families $(C \to S, \eta, \beta)$ of genus $g$ Prym curves over $S$ and a morphism $(C \to S, \eta, \beta) \to (C' \to S', \eta', \beta')$ is the data of a cartesian diagram 
    \begin{equation}\label{eqn: morphism prestack}
    \begin{tikzcd}
    C\arrow{r}{\varphi} \rar\dar\drar[phantom, "\square"] & C' \arrow{d} \\%
    S \arrow{r}{f} & S'
    \end{tikzcd}
    \end{equation}
    such that there exists an isomorphism $\tau: \varphi^* \eta' \to \eta$ \footnote{Observe that we are adopting the convention that the datum of $\tau$ is not included in the
    definition of a morphism in $\mathcal R_g^{\mathrm{pre}}$. }such that the diagram 
    \begin{equation}\label{eqn: diagram of sheaves}
    \begin{tikzcd}
    \varphi^* \eta'^{\otimes 2}\arrow{r}{\tau^{\otimes 2}} \arrow{d}{\varphi^*(\beta'^{\otimes 2})} &  \eta^{\otimes 2} \arrow{d}{\beta^{\otimes 2}} \\%
    \varphi^* \mathcal{O}_{C'} \arrow{r} & \mathcal{O}_C
    \end{tikzcd}
    \end{equation}
    commutes.

    The moduli stack $\R_g$ is the stackification of $\R_g^{\mathrm{pre}}$.
\end{definition}

We will discuss some properties of $\R_g$ is \S\ref{sec: background facts about Rg}.

\begin{remark}\label{remark: stackification}
    Stackification is a two step process (see \cite[Section 2.5.6]{Alper} for the details): given a prestack $\mathcal{Y}$, one first defines another prestack $\mathcal{Y}^{\mathrm{st}_1}$ satisfying the condition that morphisms of $\mathcal{Y}^{\mathrm{st}_1}$ glue and a morphism $\mathcal{Y} \to \mathcal{Y}^{\mathrm{st}_1}$ such that 
    $$
    \mathrm{Mor}(\mathcal{Y}^{\mathrm{st}_1}, \mathcal{Z}) \to \mathrm{Mor}(\mathcal{Y}, \mathcal{Z})
    $$
    is an equivalence of categories for all prestacks $\mathcal{Z}$ in which morphisms glue. Then, one define the stackification $\mathcal{Y}^{\mathrm{st}}$ with a morphism $\mathcal{Y}^{\mathrm{st}_1} \to \mathcal{Y}^{\mathrm{st}}$ having the same universal property as above but where $\mathcal{Z}$ is required to be a stack.

    In our paper, we will use the notation $\R_g^{\mathrm{st}_1}$ to denote the resulting prestack after the first step. Unravelling the definitions, we see that, its objects are the same as the objects of $\R_g^{\mathrm{pre}}$, but a morphism $(C \to S, \eta, \beta) \to (C' \to S', \eta', \beta')$ is the data of a cartesian diagram as in \eqref{eqn: morphism prestack}
    such that there exists an étale cover $\widetilde{S} \to S$ and a isomorphism $\tau: h^*\varphi^* \eta' \to h^* \eta$ such that the diagram \eqref{eqn: diagram of sheaves} commutes after pullback under $\widetilde{S} \times_S C \xrightarrow{h} C$.
    
    Finally, we note that, after stackification, an object of $\R_g$ still includes the data of a genus $g$ curve, and, given objects $(C/S,\eta,\beta)$ and $(C'/S',\eta',\beta')$ of $\R_g$, one has an inclusion
    \begin{equation}\label{eqn: Rg to Mg is representable }
        \mathrm{Mor}_{\R_g}( (C/S,\eta,\beta) , (C'/S',\eta',\beta')) \subseteq \mathrm{Mor}_{\mathcal{M}_g}( C/S , C'/S').
    \end{equation}
\end{remark}

\subsubsection{Background facts about \texorpdfstring{$\mathcal{R}_g$}{Rg}}\label{sec: background facts about Rg}

As we pointed out in the introduction, there is another natural stack with the same coarse moduli space as $\mathcal R_g$. The aim of this section is to shortly clarify the differences and the relations between these two stacks. 

\begin{definition}\label{Definition: Rg tilde}
    Let $\widetilde{\mathcal{R}}_g$ be the stack over the big \'etale site whose objects are families of genus $g$ Prym curves, and a morphism $(C \to S, \eta, \beta) \to (C' \to S', \eta', \beta')$ is a Cartesian diagram of curves as in \eqref{eqn: morphism prestack}
    and an isomorphism $\tau: \varphi^* \eta' \to h^* \eta$ such that the diagram \eqref{eqn: diagram of sheaves}
    commutes.
\end{definition}

It is possible to show that $\widetilde{\mathcal R}_g$ is a smooth DM-stack. We will use this throughout.

\begin{remark}
    The stack $\widetilde{\R}_g$ has a $\mathbb{Z}/2\mathbb{Z}$-2-structure (see \cite[Appendix C]{AGV} for the defintion) given by the multiplication by $-1$ on the line bundle $\eta$, and the corresponding rigidification is exactly $\R_g$.
\end{remark}

Let $\widetilde{\mathcal{J}}_g$ be the stack of smooth curves of genus $g$ and line bundles of degree $0$ with morphisms given by maps between the curves and the line bundles. Then, $\widetilde{\mathcal{J}}_g$ has a $\mathbb G_m$-2-structure, whose rigidification we denote by $\mathcal{J}_g$. The resulting stack is projective over $\mathcal M_g$ by \cite[Theorem 4.3.]{Deligne}. 

Note that there is a natural morphism $\tilde{\iota}: \widetilde{\mathcal{R}}_g \to \widetilde{\mathcal{J}}_g$.

\begin{proposition}
    With notation as above, $\mathcal{R}_g$ is a smooth DM stack, with a representable, finite and \'etale map to $\mathcal M_g$ of degree $2^{2g}-1$. Moreover, there is a commutative diagram
    $$
    \begin{tikzcd}
    \widetilde{\mathcal{R}}_g\arrow[r, "\tilde{\iota}"]\arrow[d, "\pi"] &\widetilde{\mathcal{J}}_g\arrow[d]\\
        \mathcal{R}_g \arrow[r, "\iota"]&\mathcal{J}_g
    \end{tikzcd}
    $$
    where the vertical arrows are the rigidification morphisms, and the lower horizontal arrow is a closed immersion.
\end{proposition}

The previous proposition might be well-known to experts. However, we included a sketch of the proof due to our lack of knowledge of a precise reference.

\begin{proof}[Sketch of proof] 
    Since $\widetilde{\mathcal R}_g$ is a smooth DM-stack, so is $\R_g$, by \cite[Theorem 5.1.5.]{ACV}. Moreover, the morphism $\iota: \mathcal{R}_g \to \mathcal J_g$ is obtained from $\tilde{\iota}$ and the universal property of rigidification \cite[Theorem 5.1.5.]{ACV}.
    
    Next we show that $\iota$ is a closed embedding.
    Equation \eqref{eqn: Rg to Mg is representable } shows that $\mathcal R_g \to \mathcal M_g$ is representable (by algebraic spaces), so $\iota$ is also representable.
    
    Now, we show that it is proper.
    Recall that, from Remark \ref{remark: stackification}, an $S$-object of $\mathcal R_g$ is a curve $C \to S$ of genus $g$, and the data $(\eta, \beta)$ of a line bundle $\eta$ and trivialization of $\eta^{\otimes 2}$ defined only after an \'etale base change $\widetilde{S} \to S$. Let $S = \operatorname{Spec} (R)$ be the spectrum of a DVR, and consider a smooth curve $C \to S$ of genus $g$ with the data of two Prym curves $(\eta_i, \beta_i)$ for $i=1,2$, defined only up to an \'etale cover $\widetilde{S} \to S$, that agree on the generic fiber. To check the uniqueness part of the valuative criterion, we have to show that the identity $C \to C$ is a morphism between $(C,\eta_1, \beta_1)$ and $(C, \eta_2, \beta_2)$. Because $\mathcal J_{g}\to \mathcal M_g$ is separated, up to a further \'etale cover, there is an isomorphism $\tau:\eta_1 \to \eta_2$. Finally, perhaps after a further \'etale cover, we can also arrange $\tau$ in a way that makes the diagram \eqref{eqn: morphism prestack} commute (see the proof of Proposition \ref{prop: essential surjectivity} below for more details on how to arrange $\tau$ in this way).
    
    To check the existence part of the valuative criterion, let $C\to S= \mathrm{Spec}(R)$ be as before, and let $(C,\eta,\beta)$ be a $\operatorname{Frac}(R)$-point of $\mathcal R_g$. Since $\mathcal J_g$ is proper, there is an extension of DVR's $R \to R_1$ such that $\eta$ is in fact a line bundle on the generic point of $C_{R_1}$ which can be extended to the whole $C_{R_1}$. Since the Picard group of a curve over a DVR and over its generic point are isomorphic, the extension is necessarily $2$-torsion too, and this shows the existence of an extension of $\beta$.
    
    It is clear that $\kappa: \mathcal R_g \to \mathcal M_g$ is quasi-finite. We conclude that it is representable by schemes by \cite[03XX]{Stackproj}, and finite by the Zariski Main Theorem.
    
    Moreover, the geometric fibers of $\kappa$ have constant cardinality equal to $2^{2g}-1$ and, since the two stacks are smooth, this is also the degree of $\kappa$, and this implies that all fibers are reduced. Therefore $\kappa$ is finite \'etale. 
    
    We also know that $\iota$ is representable by algebraic spaces, quasi-finite, and it is proper because $\kappa$ is, so it is representable by schemes and finite by the Zariski Main Theorem. In fact, it is injective on geometric points. Since $\kappa$ is unramified, so is $\iota$, and so $\iota$ is a closed embeding by \cite[04DG]{Stackproj}.
\end{proof} 

\begin{remark}
    The stack $\widetilde{\mathcal R}_g$ is also proper over $\mathcal M_g$ (as it has the same coarse moduli space as $\mathcal R_g$), but it is not representable over $\mathcal M_g$ and does not embed into $\widetilde{\mathcal{J}}_g$. In fact, $\widetilde{\mathcal R}_g$ is DM but $\widetilde{\mathcal J}_g$ has 1-dimensional automorphism groups.  
\end{remark}

From this perspective, the stack $\R_g$ is better behaved than $\widetilde{\mathcal R}_g$. Both variants already appeared in the literature.  The paper \cite{Cor2} discusses $\widetilde{\mathcal{R}}_g$, while \cite{Cor, BCF} focus on $\mathcal{R}_g$.

\begin{remark}
    There is nothing special about $2$-torsion in this argument. For any $n$ that is invertible over the base field, if one studies curves with a line bundle of order $n$, there are two stacks $\mathcal M_g(n)$ and $\widetilde{\mathcal M}_g(n)$. The first one is the $\mathbb Z / n \mathbb{Z}$-rigidification of the latter, and embeds into $\mathcal J_g$.
\end{remark}

\begin{remark}
    As A. Landi pointed out to us, there is a canonical map $\mathcal{H}^w_{g,2} \to \R_g$ from the moduli stack of hyperelliptic curves with $2$-Weierstrass sections (see \cite{EdidinII}), which is $\mathbb{Z}/2\mathbb{Z}$-torsor. This provides a further presentation of $\R_g$.
\end{remark}

\subsection{Preliminaries on \texorpdfstring{$G$}{G}}

We start with some notation. Recall that $G = (\mathbb{G}_m \times \mathbb{G}_m) \rtimes \Z/ 2 \Z$.

\begin{notation}\label{Notation}
    We will often think of $G$ inside $\mathrm{GL}_2$ embedded as the subgroup of matrices of the form 
    $$
    (a,b;0):=
    \begin{pmatrix}
    a & 0 \\
    0 & b 
    \end{pmatrix}
    \ \text{or} \ 
    (a,b;1):=
    \begin{pmatrix}
    0 & a \\
    b & 0 
    \end{pmatrix}
    $$
    for $a,b \in k^*$. 
\end{notation}

The next proposition explains how the group $G$ arises in our computation. 

\begin{proposition}\label{prop: identification of G}
    The group scheme $G$ is isomorphic to the group scheme 
    $$
    \underline{\mathrm{Aut}}_{\{0,\infty\}}(\P,\O_{\P}(-3))
    $$
    of automorphisms of $(\P,\O_{\P}(-3))$ that preserve the set $\{0,\infty\} \subseteq \P$.
\end{proposition}

\begin{proof}
    From the exact sequence of group schemes
    $$
    1 \to \mu_3 \to \underline{\mathrm{Aut}}(\P,\O_{\P}(-1)) \to \underline{\mathrm{Aut}}(\P,\O_{\P}(-3)) \to 1
    $$
    and identifying $\underline{\mathrm{Aut}}(\P,\O_{\P}(-1)) = \mathrm{GL}_2$, we obtain an isomorphism
    \begin{equation}\label{iso of group schemes 1}
    \underline{\mathrm{Aut}}(\P,\O_{\P}(-3)) \xrightarrow{\sim} \mathrm{GL}_2/\mu_3.
    \end{equation}
    where $\mu_3 \subseteq \mathrm{GL}_2$ is the group of diagonal matrices of the form $\mathrm{diag}(\zeta,\zeta)$ where $\zeta \in \mathbb{G}_m$ is a $3$-rd root of unity. Moreover, we have an isomorphism of group schemes
    \begin{align}\label{iso of group schemes 2}
    \mathrm{GL}_2/\mu_3 & \to  \mathrm{GL}_2 \\
    [A] & \mapsto \mathrm{det}(A) A
    \end{align}
    with inverse given by 
    $$
    B \mapsto \bigg[ \frac{B}{\mathrm{det}(B)^{\frac{1}{3}}} \bigg].
    $$
    Note that cubic roots are well-defined mod $\mu_3$. By a direct computation,
    $$
    \underline{\mathrm{Aut}}_{\{0,\infty\}}(\P,\O_{\P}(-1)) \subset \underline{\mathrm{Aut}}(\P,\O_{\P}(-1))= \mathrm{GL}_2
    $$
    is identified with $G \subseteq \mathrm{GL}_2$ and the restriction of the isomorphism \eqref{iso of group schemes 1}, yields
    $$
    \underline{\mathrm{Aut}}_{\{0,\infty\}}(\P,\O_{\P}(-3)) \xrightarrow{\sim} G / \mu_3
    $$
    while then the restriction of \eqref{iso of group schemes 2}
    $$
    G / \mu_3 \xrightarrow{\sim} G.
    $$
    The composition of the two morphisms is the desired isomorphism.
\end{proof}

We will need to know the Chow ring of $BG$, which was computed by E. Larson \cite{Eric}.

\begin{notation}\label{notation: chern classes}
    Following \cite{Eric}, we set 
    \begin{align*}
        &\beta_i= c_i(V_G) \in \CH^*(BG), \\
        & \gamma= c_1(\Gamma_G) \in \CH^*(BG),
    \end{align*}
    where $V_G$, $\Gamma_G$ are the vector bundles over $BG$ associated to the representations $V$ and $\Gamma$ defined in Definition \ref{notation: representations}.
\end{notation}

\begin{theorem}\cite[Theorem 5.2.]{Eric}\label{thm: chow BG}
    The Chow ring of $BG$ is given by
    $$
    \CH^*(BG) \cong \frac{\Z[\beta_1, \beta_2, \gamma]}{(2\gamma, \gamma(\gamma +\beta_1))},
    $$
\end{theorem}

\subsection{Construction of the morphism}

We start with defining a morphism from the quotient prestack

\begin{equation} \label{eqn: def morphism}
    \alpha^{\mathrm{pre}}:\mathcal{X}^{\mathrm{pre}}:=\bigg[ \frac{\mathrm{Sym}^4 (V^\vee) \otimes \mathrm{det}(V) \otimes \Gamma \smallsetminus \Delta}{G}\bigg]^{\mathrm{pre}} \to \R_2.
\end{equation}
as follows.

The map \eqref{eqn: def morphism} is defined at the level of objects as follows. 
Given a morphism $F:S \to \mathrm{Sym}^4 (V^\vee) \smallsetminus \Delta$, we map $F$ to the triple $(C/S,\eta, \beta)$ where 
\begin{equation}\label{eqn: def C}
C = \underline{\mathrm{Spec}}_{\mathbb{P}^1_S}(\mathcal{O}_{\mathbb{P}^1_S} \oplus \mathcal{O}_{\mathbb{P}^1_S}(-3)) 
\end{equation}
where the $\mathcal{O}_{\mathbb{P}^1_S}$-algebra structure on $\mathcal{O}_{\mathbb{P}^1_S} \oplus \mathcal{O}_{\mathbb{P}^1_S}(-3)$ is given by 
$$
 \mathcal{O}_{\mathbb{P}^1_S}(-3) \otimes \mathcal{O}_{\mathbb{P}^1_S}(-3) \xrightarrow{XY.F} \mathcal{O}_{\mathbb{P}^1_S}. 
$$
where $XY: S \to \mathrm{Sym}^2(V^\vee)$ is the constant function equal to $XY$ and $XY.F$ denotes its product with the polynomial $F$. Clearly the curve $C$ comes with a map
\begin{equation}\label{eqn: map q}
q: C \to \P_S.
\end{equation}

We have two sections $\sigma_0,\sigma_{\infty}: S \to C$, corresponding to the $0$ and $\infty$ section of $\mathbb{P}^1_S \to S$. As explained in \cite[Section 1]{Pernice}, given a section $s: S \to \P_S$ and an $\O_{\P_S}$-algebra $\mathcal{A}$, one has the functorial bijective map
$$
\mathrm{Hom}_{\P_S}(S, \underline{\mathrm{Spec}}_{\mathbb{P}^1_S}(\mathcal{A})) \to \mathrm{Hom}_{\O_S \mathrm{-alg}}(s^*(\mathcal{A}),\O_S).
$$
Note now that 
$$
\mathrm{Hom}_{\O_S \mathrm{-alg}}(0^*(\O_{\P_S} \oplus \O_{\P_S}(-3)),\O_S) \subseteq \mathrm{Hom}_{\O_S \mathrm{-mod}}(0^* \O_{\P_S}(-3), \O_S)
$$
is the subset of $\O_S$-module homomorphisms $j_0$ such that the two maps $j_0^{\otimes 2}: 0^* \O_{\P_S}(-6) \to \O_S$ and $XY.F:0^* \O_{\P_S}(-6) \to \O_S$ coincide. But the last morphism is clearly $0$, thus we can take $j_0=0$. In conclusion, the section of $\sigma_0$ is no further data. Similarly, the section $\sigma_\infty$ is no further data.

We set 
\begin{equation}\label{eqn: def eta}
    \eta= q^* \O_{\P_S}(1) \otimes \O_C(-\sigma_0-\sigma_\infty)
\end{equation}
and finally the isomorphism

\begin{equation}\label{eqn: def trivialization}
    \eta^{\otimes 2} \xrightarrow{\sim} \O_C
\end{equation}
is given by $XY$. 

\begin{remark}\label{rmk: iso O(1)=omega}
    As explained in \cite[proof of Proposition 3.1]{Vistoli}, in our situation, an isomorphism $q^* \O_{\P_S}(1) \cong \omega_{C/S}$ is functorially the same data as an isomorphism 
    $$
    \omega_{\P_S/S}(-1) \otimes \O_{\P_S}(3) \cong \O_{\P_S}.
    $$ 
    We have such an isomorphism, up to fixing once and for all an isomorphism $\omega_{\P} \cong \O_{\P}(-2)$. We will fix such an isomorphism.
\end{remark}

The next lemma will be useful later.

\begin{lemma}\label{lemma: uniqueness of sigma0 u sigmainfty}
    The union $\sigma_0(S) \cup \sigma_\infty (S) \subseteq C$ is the unique effective relative Cartier divisor of $\omega_{C/S} \otimes \eta^\vee$.
\end{lemma}
\begin{proof}
    The restriction of $\omega_{C/S} \otimes \eta^\vee$ to each geometric fiber $C_s$ is a degree $2$ line bundle not isomorphic to $\omega_{C_s}$, thus $h^0(C_s,\omega_{C/S} \otimes \eta^\vee)=1$ for all $s \in S$. 
\end{proof}

Now, we define the morphism \eqref{eqn: def morphism} at the level of morphisms. A morphism 
$$
(F: S \to \mathrm{Sym}^4 (V^\vee) \otimes \mathrm{det}(V) \otimes \Gamma \smallsetminus \Delta) \to (F': S' \to \mathrm{Sym}^4 (V^\vee) \otimes \mathrm{det}(V) \otimes \Gamma \smallsetminus \Delta)
$$
in $\mathcal{X}^{\mathrm{pre}}$ is by definition the data of morphisms $f: S \to S'$ and $g:S \to G$ such that $F' \circ f = g. F$, where $g . F$ denotes the action of $g$ on $F$.

\begin{remark}
    Write $g(s)= (a(s),b(s),\sigma^i(s))$ for $s \in S$. Then, spelling out the action of $G$ on $F$, this is equivalent to 
    \begin{equation}\label{eqn: action on F}
        F'(f(s))= a(s) b(s) \bigg( F \circ (a(s),b(s),\sigma^i(s))^{-1} \bigg)
    \end{equation}
for $s \in S$.
\end{remark}

Let $(C/S, \eta,\beta)$ and $(C'/S', \eta',\beta')$ be the objects in $\R_2$ associated to $F$ and $F'$ respectively. Call also 
$$
\phi= f \times g: \P_S \to \P_{S'}.
$$ 
Here, we view $g$ as an automorphism of $\P$ via the isomorphism in Proposition \ref{prop: identification of G}. We have a morphism 
\begin{equation}\label{eqn: def map on morphism part 1}
\varphi: C \to C'
\end{equation}
given by
$
\varphi^{\#}: \phi^* \O_{\P_{S'}}(-3) \to \O_{\P_{S}}(-3)
$
defined at $(s,\ell) \in \P_S=S \times \P$ by 
$$
 (\varphi^{\#}(\phi^*(v))) (x,\ell)= g(s)^{-1}. \ v (f(s),g(s). \ell)
$$

Again, $G$ acts on $(\P, \O(-3))$ via the isomorphism in Proposition \ref{prop: identification of G}.

\begin{lemma}
    The morphism
    $$
    \phi^*(\O_{\P_{S'}} \oplus \O_{\P_{S'}}(-3)) \to \O_{\P_{S}} \oplus \O_{\P_{S}}(-3)
    $$
    induced by $\varphi^{\#}$ is an homomorphism of $\O_{\P_S}$-algebras.
\end{lemma}

\begin{proof}
    Here it is crucial to take into account how $G$ acts on the space $\mathrm{Sym}^4(V^\vee) \smallsetminus \Delta$. Since the algebra structure on $\mathcal{O}_{\P} \oplus \mathcal{O}_{\P}(-3)$ is determined by $XY. F$, the statement is equivalent to the identity
    \begin{equation}\label{eqn: varphi is homo of algebars}
    ((\varphi^{\#}) ^\vee)^{\otimes 2} \bigg( \phi^*(XY F') \bigg) = XY F
    \end{equation}
    as sections of $\O_{\P_S}(6)$. We check this at a point $(s,\ell) \in \P_S= S \times \P$:
    \begin{align*}
    ((\varphi^{\#}) ^\vee)^{\otimes 2} \bigg( \phi^*(XY F') \bigg)(s,\ell)&= g(s)^{-1}. \bigg( (XYF')(f(s),g(s). \ell) \bigg) \\
    &= (XYF') \circ \bigg( \frac{(a(s),b(s);\sigma^i(s))}{\mathrm{det}(a(s),b(s);\sigma^i(s))^{1/3}} \bigg) (s,\ell) \\
    &=\bigg( a(s)^{-1} b(s)^{-1} (XY) (F' \circ (a(s),b(s);\sigma^i(s))) \bigg) (s,\ell) \\
    &= (XY F)(s,\ell)
    \end{align*}
    where we wrote $g(s)=(a(s),b(s);\sigma^i(s)) \in G(S)$ with $\sigma^i(s)=\delta_0^i$ as explained in Notation \ref{Notation}. Note that in the last equality we used Equation \eqref{eqn: action on F}.
\end{proof}

This shows that the morphism \eqref{eqn: def map on morphism part 1} is well-defined. 

To conclude the definition of $\alpha^{\mathrm{pre}}$ on morphisms we need to provide, étale locally on $S$, a morphism $\tau$ between $\varphi^* \eta'$ and $\eta$ making the diagram \eqref{eqn: diagram of sheaves} commute. Note that, since $g$ preserves $\{0,\infty\}$, we have
$$
\varphi^{-1} (\sigma_0'(S') \cup \sigma_\infty' (S') )= \sigma_0(S) \cup \sigma_\infty (S)
$$
and thus 
\begin{equation}\label{eqn: def morphism 3}
\varphi^* (-\sigma_0'(S') - \sigma_\infty' (S'))= -\sigma_0(S) - \sigma_\infty (S)
\end{equation}
as Weil divisors. Furthermore, étale locally on $S$, we have an isomorphism
\begin{equation}\label{eqn: def morphism 4}
\phi^* \O_{ \P_{S'}}(1) \xrightarrow{\sim} \O_{\P_S}(1)
\end{equation}
that over a point $(s,\ell) \in \P_S$  maps $\varphi \in (g(s). \ell)^\vee$ to $\varphi \circ \frac{ g(s) }{\sqrt{a(s) b(s)}} \in \ell^\vee$. A square root of $ab$ is defined étale locally on $S$. We claim that we can take $\tau$ to be the product of the morphisms in \eqref{eqn: def morphism 3} and \eqref{eqn: def morphism 4}. 

\begin{lemma}\label{lemma: def of tau}
    Let $\tau$ be the isomorphism $\varphi^* \eta' \to \eta$ obtained as the product of the morphisms in \eqref{eqn: def morphism 3} and \eqref{eqn: def morphism 4}. Then, the diagram \eqref{eqn: diagram of sheaves} commutes.
\end{lemma}

\begin{proof}
    This is an easy check left to the reader.
\end{proof}

This concludes the definition of the morphism \eqref{eqn: def morphism}.

\subsection{Proof of Theorem \ref{thm: presentation R2}} \label{sec: proof of presentation}

By the universal property of stackification, the morphism \eqref{eqn: def morphism} yields a morphism
\begin{equation} \label{eqn: def morphism of stacks}
    \alpha: \mathcal{X}:=\bigg[ \frac{\mathrm{Sym}^4 (V^\vee) \otimes \mathrm{det}(V) \otimes \Gamma \smallsetminus \Delta}{G}\bigg] \to \R_2
\end{equation}
which we aim to prove is an isomorphism of stacks.

We will use the following well-known lemma. 

\begin{lemma}\label{lemma: stacky lemma}
Let $b : \mathcal{X}^{\mathrm{pre}} \to \mathcal{X}$ be a stackification. Let $\alpha^{\mathrm{pre}} : \mathcal{X}^{\mathrm{pre}} \to \mathcal{Y}$ be a morphism to a stack $\mathcal{Y}$, and let $\alpha: \mathcal{X} \to \mathcal{Y}$ be the induced morphism. If $\alpha^{\mathrm{pre}}$ satisfies the following two conditions:
\begin{itemize}
    \item[(i)] it is fully faithful,
    \item[(ii)] for every object $Y \in \mathcal{Y}(S)$ over $S$ there is a covering $\widetilde{S} \to S$ and an object $\widetilde{X} \in \mathcal{X}^{\mathrm{pre}}(\widetilde{S})$ such that $\alpha^{\mathrm{pre}}(\widetilde{X})= Y_{|\widetilde{S}}$,
\end{itemize}
then $\alpha$ is an isomorphism.
\end{lemma}

We start with showing that $\alpha^{\mathrm{pre}}$ is fully faithful.

\begin{proposition}\label{prop: fully faithfullness}
    The morphism $\alpha^{\mathrm{pre}}$ satisfies condition $(i)$ of Lemma \ref{lemma: stacky lemma}.
\end{proposition}

\begin{proof}
    First, we notice that $\alpha^{\mathrm{pre}}$ factors as 
    $$
    \alpha^{\mathrm{pre}}: \mathcal{X}^{\mathrm{pre}} \to \mathcal R_2^{\mathrm{st}_1} \to \mathcal R_2.
    $$
    Since the second map is fully faithful by construction, we only need to prove that the first map is faithful. We will refer to again by the symbol $\alpha^{\mathrm{pre}}$.
    It is enough to prove that $\alpha^{\mathrm{pre}}(S): \mathcal{X}^{\mathrm{pre}}(S) \to \R_2^{\mathrm{st}_1}(S)$ is fully faithful for all schemes $S$. Let $u,u':S \to \mathcal{X}^{\mathrm{pre}}$ be objects in $\mathcal{X}$ and let $\alpha^{\mathrm{pre}}(u)=(\pi: C \to S,\eta,\beta)$ and $\alpha^{\mathrm{pre}}(u')=(\pi':C' \to S,\eta',\beta')$ be their images under $\alpha^{\mathrm{pre}}$. We wish to show that the induced map
    \begin{equation}\label{eqn: fully faithfullness}
    \mathrm{Mor}_{\mathcal{X}(S)}(u,u') \to \mathrm{Mor}_{\R_2^{\mathrm{st}_1}(S)}((C/S,\eta,\beta),(C'/S,\eta',\beta'))
    \end{equation}
    is bijective.

    Let $\varphi: C \to C'$ be an isomorphism for which there exists an isomorphism $\tau: \varphi^* \eta' \to \eta$ such that the diagram \eqref{eqn: diagram of sheaves} commutes.
    First of all, we note that there exists a unique isomorphism $\phi: \P_S \to \P_S$ such that the following diagram 
    \begin{equation}\label{eqn: required commutativity }
    \begin{tikzcd}
    C\arrow{r}{\varphi} \arrow{d}{q} & C'\arrow{d}{q'} \\%
    \P_S \arrow{r}{\phi} & \P_S
    \end{tikzcd}
    \end{equation}
    commutes. The uniqueness of $\phi$ is clear.
    To show the existence of $\phi$, we consider the corresponding diagram where the vertical arrows are the canonical maps and the lower horizontal arrow is given by the pullback of forms via $\varphi^*$
    \[
    \begin{tikzcd}
    C\arrow{r}{\varphi} \arrow{d}{p} & C'\arrow{d}{p'} \\%
    \Pp (\pi_* \omega_{C/S})^\vee \arrow{r}{\varphi^*} & \Pp (\pi'_* \omega_{C'/S})^\vee.
    \end{tikzcd}
    \]
    This diagram clearly commutes. 

    By Remark \ref{rmk: iso O(1)=omega}, the maps $q$ and $q'$ are given by two sections of $\omega_{C/S}$ and $\omega_{C'/S}$ respectively. These two sections yields commutative diagrams 
    \begin{equation*}\label{eqn: commutative triangles}
    \adjustbox{scale=0.95,center}{
        \begin{tikzcd}
        & C  \arrow{dr}{q} \arrow{dl}{p}  &  &   & C'  \arrow{dl}{p'} \arrow{dr}{q'}\\
        \Pp (\pi_* \omega_{C /S})^\vee \arrow{rr}{h} & & \P_{S} & \Pp (\pi'_* \omega_{C' /S})^\vee \arrow{rr}{h'} & & \P_{S}
        \end{tikzcd}
    }
    \end{equation*}
    where $h$ and $h'$ are isomorphisms. In these terms, the morphism $\phi$ in diagram \eqref{eqn: required commutativity } is the composition $h' \circ \varphi^* \circ h^{-1}: \P_{S} \to \P_{S}$.
    
    At this point, the morphism $\varphi$ induces an isomorphism 
    $$
    \varphi^{\#}: \phi^* \O_{\P_S}(-3) \to \O_{\P_S}(-3).
    $$
    This shows that $\varphi$ can be the image of a unique element $ g \in \underline{\mathrm{Aut}}(\P_S, \O_{\P_S}(-3))$ and injectivity of \eqref{eqn: fully faithfullness} follows. 

    In order to prove surjectivity of \eqref{eqn: fully faithfullness}, it is enough to check that $g \in \underline{\mathrm{Aut}}_{\{ 0,\infty\}}(\P_S, \O_{\P_S}(-3))$. The existence of $\tau$ implies that there is an isomorphism
    $$
    \varphi^* \O_{C'}(\sigma_0' + \sigma_\infty') \cong \O_{C}(\sigma_0 + \sigma_\infty).
    $$
    and, from Lemma \ref{lemma: uniqueness of sigma0 u sigmainfty}, we obtain that
    $$
    \varphi^{-1}(\sigma_0'(S) \cup \sigma_\infty'(S))=\sigma_0(S) \cup \sigma_\infty(S).
    $$
    Finally, the commutativity of \eqref{eqn: required commutativity } implies that $\phi(0(S) \cup \infty(S))=0'(S) \cup \infty'(S)$ and thus that $g \in \underline{\mathrm{Aut}}_{\{ 0,\infty\}}(\P_S, \O_{\P_S}(-3))$. This concludes the proof.
\end{proof}

\begin{proposition}\label{prop: essential surjectivity}
    The morphism $\alpha^{\mathrm{pre}}$ satisfies condition $(ii)$ of Lemma \ref{lemma: stacky lemma}.
\end{proposition}

\begin{proof}[Proof of Proposition \ref{prop: essential surjectivity}]
    Let $X$ be an object of $\R_2$ over a scheme $S$. Up to an étale cover of $S$, we may assume that $X$ is of the form $X=(C'/S,\eta', \beta': \eta'^{\otimes 2} \xrightarrow{\sim} \O_C)$. We need to show that, up to an étale cover of $S$, there exists $F:S \to \mathrm{Sym}^4 (V^\vee) \otimes \mathrm{det}(V) \otimes \Gamma \smallsetminus \Delta$ and:
    \begin{enumerate}
        \item[(a)] an isomorphism $\varphi$ between the curve $C$ in \eqref{eqn: def C} and $C'$;
        \item[(b)] an isomorphism $\tau: \varphi^*\eta' \to \eta$ where $\eta$ is the line bundle defined in \eqref{eqn: def eta} such that the diagram \eqref{eqn: diagram of sheaves} commutes.
    \end{enumerate}

Let $\bar{s}$ be a geometric point of $S$. We will prove that (a) and (b) holds in an étale neighborhood of $\bar{s}$. By \cite{Vistoli}, we know that, perhaps after base-change by an étale cover of $S$, there exists $F$ in $\mathcal{X}^{\mathrm{pre}}$ over $S$ such that there is an isomorphism $\varphi: C \to C'$ over $S$. By Lemma \ref{lemma: idea}, we have $\eta'_{\bar{s}} = q_{\bar{s}}'^* \O(1) \otimes \O_{C'_{\bar{s}}}(-w_i-w_j)$ for exactly one subset $\{ w_i , w_j \}$ of cardinality $2$ of the set of Weierstrass points of $C'_{\bar{s}}$. Here $q'_{\bar{s}}: C'_{\bar{s}} \to \mathbb{P}^1$ is any degree $2$ map such that $q_{\bar{s}}'(\{ w_i , w_j \})= \{ 0, \infty \} \subseteq \mathbb{P}^1$. Up to composing the map \eqref{eqn: map q} with an automorphism of $\P$ fixing $\{0,\infty \}$ (which amounts to changing $F$ -- but not $C$), we may assume that 
$$
q_{\bar{s}}' \circ \varphi_{\bar{s}} = q_{\bar{s}}
$$
where $q$ is the map in \eqref{eqn: map q}. It follows then $\varphi^* \eta'_{\bar{s}} = \eta_{\bar{s}}$. The locus 
$$
\bigl\{ s \in S \ | \ \varphi^* \eta'_s \cong q_s^* \O_{\P_S}(1) \otimes \O_{C_s}(- \sigma_0(s) -\sigma_\infty(s))\bigl\}
$$
is, by Lemma \ref{lemma: idea}, open and closed in $S$ and contains ${\bar{s}}$. We may thus replace $S$ with it. It follows that 
$$
\varphi^* \eta'= \eta \otimes L
$$
where $L$ is a line bundle pulled back to $C$ from $S$. Zariski-locally on $S$, the line bundle $L$ is trivial. Up to replacing $S$ by an open subset containing $\bar{s}$ we may then assume that we have an isomorphism $\varphi^* \eta'  \xrightarrow{\tau} \eta$. The diagram 
$$
\begin{tikzcd}
    \varphi^* \eta'^{\otimes 2}\arrow{r}{\tau^{\otimes 2}} \arrow{d}{\varphi^*(\beta'^{\otimes 2}))} & \eta^{\otimes 2} \arrow{d}{\beta^{\otimes 2}} \\%
    \varphi^* \mathcal{O}_{C'} \arrow{r} & \mathcal{O}_C
    \end{tikzcd}
$$
commutative up to some $\lambda \in \mathbb{G}_m(S)$. Here $\beta$ is defined in \eqref{eqn: def trivialization}. However, étale locally around $\bar{s}$, we can find a square root $\mu \in \mathbb{G}_m(S)$ of $\lambda$ and, replacing $\tau$ with $\frac{\tau}{\mu}$, we obtain the commutativity of the above diagram. This concludes the proof of the proposition.
\end{proof}

\begin{proof}[Proof of Theorem \ref{thm: presentation R2}]
    Theorem \ref{thm: presentation R2} now follows immediately from Lemma \ref{lemma: stacky lemma} and Propositions \ref{prop: fully faithfullness} and \ref{prop: essential surjectivity}.
\end{proof}

\section{Computation of the Chow ring of \texorpdfstring{$\R_2$}{R2}}

In this section we use Theorem \ref{thm: presentation R2} and the tools in \cite{EG} to compute the Chow ring of $\R_2$. In the following, we will repeatedly use the isomorphism $\CH^*([X/G]) = \CH^*_G(X)$ without explicitly mentioning it.

The first observation is that we have a $G$-equivariant projection

\begin{equation}\label{eqn: projectivization}
    p: \mathrm{Sym}^4 (V^\vee) \otimes \mathrm{det}(V) \otimes \Gamma \smallsetminus \Delta \to \Pp \mathrm{Sym}^4 (V^\vee) \smallsetminus \underline{\Delta} 
\end{equation}
where $ \underline{\Delta} \subseteq \Pp \mathrm{Sym}^4 (V^\vee)$ is the locus of polynomials (up to non-zero scalars) having either a root at $0$ or $\infty$, or a double root in any field extension of the base field $k$.

\begin{lemma}\label{lem: chow of Gm torsor}
    The pullback $p^*$ induces an isomorphism
    $$
    \frac{\CH_G^*( \mathbb{P} (\mathrm{Sym}^4 (V^\vee) )\smallsetminus \underline{\Delta} ) }{(-h+\beta_1+\gamma)}\xrightarrow{\sim} \CH^*_G(\mathrm{Sym}^4 (V^\vee) \otimes \mathrm{det}(V) \otimes \Gamma \smallsetminus \Delta)
    $$
    where $h=c_1^G(\mathcal{O}_{\Pp \mathrm{Sym}^4 (V^\vee)}(1))$, and $\beta_1$ and $\gamma$ are defined in Notation \ref{notation: chern classes}.
\end{lemma}

\begin{proof}
    The morphism $p$ is a $\mathbb{G}_m$-torsor with associated line bundle $\O_{\Pp \mathrm{Sym}^4 (V^\vee)}(-1) \otimes \mathrm{det}(V) \otimes \Gamma $.
\end{proof}

Next, we compute $\CH_G^*( \Pp \mathrm{Sym}^4 (V^\vee)\smallsetminus \underline{\Delta} $ using the excision sequence:

\begin{equation}\label{eqn: excision sequence}
    \CH_{*}^G(\underline{\Delta}) \to \CH_*^G ( \Pp \mathrm{Sym}^4 (V^\vee)) \to \CH^G_*(\Pp \mathrm{Sym}^4 (V^\vee)  \smallsetminus \underline{\Delta}) \to 0
\end{equation}

In order to do so, we will need a $G$-equivariant envelope of $\underline{\Delta}$. 

\subsection{A \texorpdfstring{$G$}{G}-equivariant envelope for \texorpdfstring{$\underline{\Delta}$}{delta}}\label{sec: envelope}
We first recall the definition.

\begin{definition}\cite[Definition 18.3]{Fu} and \cite[page 9]{EG}
    An envelope of a scheme $X$ is a proper morphism $p : X' \to X$
    such that for every subvariety \footnote{By variety we mean an integral scheme. A subvariety of a scheme is a closed subscheme which is a variety.} $Y$ of $X$, there is a subvariety $Y'$ of $X'$ such that $p$ maps $Y'$ birationally onto $Y$. 
    
    Suppose that $G$ is a linear algebraic group acting on $X$ and $X'$. We will say that $p : X' \to X$ is a $G$-equivariant envelope, if $p$ is $G$-equivariant, proper and if we can take $Y'$ to be $G$-invariant for $G$-invariant $Y$.  
\end{definition}

The importance of $G$-envelope is explained by the next lemma.

\begin{lemma}
    Let $p:X' \to X$ be a $G$-equivariant envelope. Then the pushforward
    $$
    p_*: \CH_*^G(X') \to \CH_*^G(X)
    $$
    is surjective.
\end{lemma}

\begin{proof}
    This follows from \cite[Lemma 3]{EG} and \cite[Lemma 18.3.(6)]{Fu}.
\end{proof}

Finally, the next lemma explains how to  construct a $G$-equivariant envelope.

\begin{lemma}\label{lemma: criterion for envelope}
    Let $X$ be a scheme and $G$ a linear algebraic group acting on $X$. Given $G$-equivariant proper and surjective morphisms $p_i: X_i' \twoheadrightarrow X_i$ for $i=1,\ldots,d$ such that:
\begin{enumerate}
    \item[(i)] $X_1=X;$
    \item[(ii)] $X_{i+1} \subseteq X_{i}$ for $i=1, \ldots,d-1$;
    \item[(iii)] if $K \supseteq k$ is any extension and $x$ is a $K$-valued point of $X_i \smallsetminus X_{i+1}$,  then there is a unique $K$-valued point $y$ of $X_i'$ mapping to $x$.
\end{enumerate}
then the map $p= \sqcup p_i: \sqcup_{i=1}^d X_{i}' \to X $ is a $G$-equivariant envelope. 
\end{lemma}

\begin{proof}
    See \cite[Proof of Proposition 4.1]{EF}. 
\end{proof}

In our situation, we take $p$ to be the disjoint union of the following maps $p_{i}$ for $i=1, \ldots,4$.

We let 
$$
p_{1,1}: \Pp V^\vee \times \Pp\mathrm{Sym}^2(V^\vee) \to \ \underline{\Delta} 
$$
defined by $ (F,G) \mapsto F^2 G $ and 
$$
p_{1,2}: \Pp \mathrm{Sym}^3(V^\vee)\times \{X,Y\} \subset \Pp \Sym^3 (V^\vee) \times \Pp V^\vee \to \ \underline{\Delta}  
$$
defined by $(F,G) \mapsto FG$, where $G \in \{X,Y\}$. Then, we set $p_1= p_{1,1} \sqcup p_{1,2}$ and $\underline{\Delta}_1 = \underline{\Delta}$ to be the image of $p_1$.

Let 
$$
p_{2,1}: \Pp \mathrm{Sym}^2(V^\vee) \to \ \underline{\Delta} 
$$
be defined by $ F \mapsto F^2$,
$$
p_{2,2}: \Pp V^\vee  \times \Pp V^\vee  \times \{X,Y\} \subset (\Pp V^\vee)^{\times 3}\to \ \underline{\Delta}
$$
defined by $(F_1,F_2,G) \mapsto F_1^2 F_2 G$, where $G \in \{X,Y\}$ and
$$
p_{2,3}: \Pp \mathrm{Sym}^2(V^\vee) \times \{X^2, Y^2, XY\} \subset (\Pp \Sym^2(V^\vee))^{\times 2}\to \ \underline{\Delta} 
$$
be defined by $(F,G) \mapsto FG$, where $G \in \{X^2, Y^2, XY\}$. We set $p_2=p_{2,1} \sqcup p_{2,2} \sqcup p_{2,3}$ and $\underline{\Delta}_2 \subseteq \underline{\Delta}_1$ to be the image of $p_2$.

Define
$$
p_{3,1}: \Pp V^\vee\times  \{X^3, Y^3, X^2Y, XY^2\} \subset \Pp V^\vee \times \Pp\Sym^3(V^\vee) \to \underline{\Delta}
$$
by $(F,G) \mapsto FG$, where $G \in \{X^3, Y^3, X^2Y, XY^2\}$, and
$$
p_{3,2}: \Pp V^\vee \times \{X^2, Y^2, XY\}\subset \Pp V^\vee \times \Pp \Sym^2(V^\vee) \to \underline{\Delta}
$$
by $(F,G) \mapsto F^2G$, where $G \in \{X^2, Y^2, XY\}$. We set $p_3=p_{3,1} \sqcup p_{3,2}$ and $\underline{\Delta}_3 \subseteq \underline{\Delta}_2$ to be the image of $p_3$.

Finally, we set 
$$
p_4: \{X^4, Y^4,X^3Y,XY^3,X^2Y^2\} \to \underline{\Delta}
$$
to be the inclusion, and $\underline{\Delta}_4$ will be its image. 

Let $\underline{\Delta}'$ be the disjoint union of the domains of $p_1, p_2, p_3$ and $p_4$.
\begin{proposition}\label{prop: p is envelope}
    The morphism 
    $$
    p=p_1 \sqcup p_2 \sqcup p_3 \sqcup p_4 : \underline{\Delta}' \longrightarrow \underline{\Delta}
    $$
    is a $G$-equivariant envelope.
\end{proposition}

\begin{proof}
    %This uses the fact that $\mathrm{char}(k)>3$ is at least $3$. 
    We will check the conditions of Lemma \ref{lemma: criterion for envelope}. Condition $(i)$ and $(ii)$ are obvious. We will check that condition $(iii)$ holds for $i=1$, the other cases are similar. Let $f \in K[X,Y]$ be a $K$ valued point of $\underline{\Delta}_1 \smallsetminus \underline{\Delta}_2$ and let $n_0$ and $n_\infty$ be respectively the multiplicities of $0$ and $\infty$ as roots of $f \in K[X,Y]$. Clearly $n_0 + n_\infty \leq 1$ otherwise $f$ would be in $\underline{\Delta}_2$. 
    
    If $n_0 + n_\infty = 1$, then $f$ cannot be in the image of $p_{1,1}$ and must be in the image of $p_{1,2}$. Suppose $n_0=1$ and $n_\infty=0$. Then, we can write $f=X F$ for a unique $F \in K[X,Y]_3$ (up to $K^*$) and $f=p_{1,2}((F,X))$. 

    If $n_0+n_\infty=0$, then clearly $f$ is not in the image of $p_{1,2}$ and must be in the image of $p_{1,1}$. In fact, by definition, $f$ must have a double root in some extension of $K$, and thus also in $K$ as by assumption $\mathrm{char}(k)$ is bigger than the degree of $f$. Write $f=F^2 G$ where $F \in K[X,Y]_1$ and $G \in K[X,Y]_2$. Then, since $f \notin \underline{\Delta}_2$, the polynomial $G$ is not a square (in any extension of $K$). Thus $(F,G)$ is uniquely determined.
\end{proof}

\subsection{Review of multiplication maps}
In order to compute the pushforward along the morphism $p$ we will make use of the machinery developed in \cite[Section 4]{Eric}. We recall here the basic definitions and results.

Let $E$ be a vector bundle of rank $2$ over an algebraic stack $\mathcal{X}$ with Chern classes $c_1, c_2 \in \CH^*(\mathcal{X})$. Denote by $s_r^j$ the pushforward of
$$
\underbrace{c_1(\mathcal O_{\Pp E}(1)) \otimes \ldots \otimes c_1(\mathcal O_{\Pp E} (1))}_{j \text{ times}} \otimes 1
$$
along the multiplication map
$$
(\Pp E)^j \times_\mathcal{X} \Pp \Sym^{r-j} (E) \to \Pp \Sym^r (E).
$$
 If $h = c_1(\mathcal O_{\Pp \Sym^r (E)}(1))$, then $s_r^0=1$, $s_r^1 = h$ and, more generally, for all $j \ge 0$ the following relation holds:
$$
s_r^{j+1} = (h + jc_1)s_r^{j} + j(r+1-j)c_2 s^{j-1}_r.
$$
So one obtains:
\begin{align}
\begin{split}\label{eqn: h^2}
    h^2 =\,&s^2_r -c_1 s^1_r - rc_2
\end{split}\\
\begin{split}\label{eqn: h^3}
    h^3 =\,&s_r^3 -3c_1 s_r^2 + (c_1^2 +(2-3r)c_2)s^1_r + rc_1c_2
\end{split}\\
\begin{split}\label{eqn: h^4}
    h^4 =\,& s_r^4 -6c_1s^3_r + (7c_1^2 -(6r-8)c_2)s_r^2 + ((10r-8)c_1c_2 -c_1^3)s_r^1 \\
    &-(rc_1^2c_2-(3r^2-2r)c_2^2)s_r^0 
\end{split}
\end{align}
where $r \geq 2$, $3$ and $4$ in \eqref{eqn: h^2}, \eqref{eqn: h^3} and \eqref{eqn: h^4} respectively. Moreover, the pushforward along multiplication maps
$$
\mathrm{mult}: \Pp\Sym^a (E) \times_\mathcal{X} \Pp \Sym^b (E) {\longrightarrow} \Pp \Sym^{a+b} (E)
$$
is obtained from the bilinear map
\begin{equation}\label{eqn: multiplication maps}
\begin{alignedat}{3}
    \mathrm{mult}_* \colon &&   \CH_*(\Pp \Sym^a(E)) \times \CH_*(\Pp \Sym^b(E)) &   \to{}   && \CH_*(\Pp \Sym^{a+b} (E)) \\
             && (s_a^\alpha, s_b^\beta) & \mapsto{} && \binom{a+b-\alpha -\beta}{a-\alpha}s_{a+b}^{\alpha +\beta}
\end{alignedat}
\end{equation}

The maps $p_{1,1}$ and $p_{2,2}$ involve squaring, so we will also need the class of the diagonal. 

\begin{lemma}\label{lem: the diagonal of Psym}
    The pushforward along the squaring map
    $$
    \mathrm{sq}: \Pp  E \longrightarrow \Pp \Sym^2 (E)
    $$
    is given by
    \begin{align*}
        s_1^0 & \mapsto 2s^1_2 + 2c_1\\
        s_1^1 &\mapsto  s_2^2-2c_2.
    \end{align*}
    The pushforward along the squaring map
    $$
    \mathrm{sq}: \Pp \Sym^2 (E) \longrightarrow \Pp \Sym^4 (E)
    $$
    is given by
    \begin{align*}
        s_2^0 & \mapsto  4s_4^2 + 12c_1s_4^1 + 12c_1^2  \\
        s_2^1 &\mapsto  2s_4^3 +2c_1 s_4^2 -12c_2 s_4^1-24 c_1c_2\\
        s_2^2 &\mapsto s_4^4  -4c_2s_4^2+24c_2^2 .
    \end{align*}
\end{lemma}
\begin{proof}
    Let $W$ denote the vector bundle $E$ or $\mathrm{Sym}^2 (E)$, and consider the tautological sequence
    $$
    0 \longrightarrow \mathcal O(-1) \longrightarrow W\longrightarrow Q\longrightarrow 0
    $$
    on $\Pp W$. The diagonal $\Delta_{\Pp W} \subset \Pp W \times_\mathcal{X} \Pp W$ is the vanishing locus of the map
    $$
    p_1^* \mathcal O(-1) \longrightarrow p_1^* W = p_2^* W \longrightarrow p_2^*Q,
    $$
    where $p_1, p_2 : \Pp W \times_\mathcal{X} \Pp W \to \Pp W$ are the projections. Therefore, its class is given by
    $$
    [\Delta_{\Pp W}] = \left[\frac{c(W)}{(1-h_1)(1-h_2)}\right]_{\mathrm{rank}(W)-1},
    $$
    where $h_i=p_i^*h$ for $i=1,2$, $h=c_1(\mathcal{O}_{\Pp W)}(1))$. Moreover, $\Delta_{\Pp W}^*(h_1)$ is again the hyperplane class, so by the projection formula,
    $$
    \Delta_{\Pp W,*}(h^k) = h_1^k \cdot[\Delta_{\Pp W}].
    $$
    If $W = E$, then $c(W) = 1+ c_1 + c_2$, so $[\Delta_{\Pp E}] = h_1+h_2 +c_1$ and
    \begin{align*}
        \Delta_{\Pp E, *} (1) &= h_1 + h_2 + c_1\\
        \Delta_{\Pp E ,*}(h) &= h_1^2 + h_1 \otimes h_2 + c_1 h_1 =h_1 \otimes h_2 -c_2
    \end{align*}
    where in the last equality we have used the projective bundle formula. If $W = \Sym^2 (E)$, $c(W) = 1+ 3c_1 + 2c_1^2 + 4c_2 + 4c_1c_2$, so
    \begin{align*}
        \Delta_{\Pp \Sym^2 (E), *} (1) &= h_1^2 + h_1 \otimes h_2 + h_2^2 + 3c_1(h_1 + h_2) + 2c_1^2 + 4c_2\\
        \Delta_{\Pp \Sym^2 (E) ,*}(h) &= h_1^3 + h_1^2 \otimes h_2 + h_1 \otimes h_2^2 + 3c_1h_1 \otimes h_2 + 3c_1 h_1^2 + (2c_1^2 + 4c_2)h_1\\
        &= h_1^2 \otimes h_2 + h_1 \otimes h_2^2 + 3c_1h_1 \otimes h_2 -4c_1c_2\\    \Delta_{\Pp \Sym^2 (E) ,*}(h^2) &=h_1^3 \otimes h_2 + h_1^2 \otimes h_2^2 + 3c_1h_1^2 \otimes h_2  -4c_1c_2 h_1\\
        &= h_1^2 \otimes h_2^2 -(2c_1^2+4c_2)h_1\otimes h_2 - 4c_1c_2(h_1 + h_2).
    \end{align*}
    where again we have used the projective bundle formula. The squaring maps are the composition $\mathrm{mult}\circ \Delta_{\Pp W}$, so the result follows from formula \eqref{eqn: multiplication maps}, after using \eqref{eqn: h^2}.
\end{proof}

For us, $E$ will be $V_G^\vee$, and $X$ will be $BG$, so $c_1 = -\beta_1$ and $c_2 = \beta_2$.

\subsection{The classes of the finite subsets}

Almost all the maps $p_{ij}$ in \S\ref{sec: envelope} are obtained as a composition
$$
\Pp \Sym^a (V_G^\vee) \times Z \hookrightarrow \Pp \Sym^a (V_G^\vee)  \times \Pp \Sym^b (V_G^\vee) \stackrel{\mathrm{mult}}{\longrightarrow} \Pp \Sym^{a+b} (V_G^\vee)
$$
for some substack $Z \subseteq  \Sym^b (V_G^\vee)$ corresponding to a finite set of points in $V^\vee$. Here we calculate the equivariant fundamental classes of such finite subsets.
\begin{lemma}
    The class of $\{X, Y\} \subset \Pp V_G^\vee$ is $2s_1^1 - (\beta_1 + \gamma)s_1^0$.
\end{lemma}
\begin{proof}
    Let $\xi: B \mathbb G_m^2 \to BG $ the map induced by $\mathbb G_m^2  \to G$, and let $t_1, t_2 \in \CH^*(B \mathbb G_m^2)$ be the Chern roots of $\xi^* V^\vee_G=V_{\mathbb G_m^2}$.  There is a Cartesian diagram
    $$
    \begin{tikzcd}
    \Pp V^\vee_{\mathbb G_m^2} \arrow[r, "\xi"] \arrow[d] & \Pp V^\vee_G \arrow[d] \\
    B\mathbb G_m^2 \arrow[r, "\xi"]                       & BG                    
\end{tikzcd}
    $$
    where $\xi^* \mathcal O (1) = \mathcal O(1)$ and $\xi_*[\{X\}] = [\{X,Y\}]$. Using \cite[Lemma 2.4.]{EF}, one sees that
    $$
    [X]^{\mathbb G_m^2} = c_1^{\mathbb G_m^2}(\mathcal O_{\Pp V^\vee} (1)) -t_1
    $$
    and so by the projection formula,
    $$
    [\{X,Y\}] = \xi_*(c_1^{\mathbb G_m^2}(\mathcal O_{\Pp V^\vee} (1)) -t_1) = \xi_*(1) c_1^G(\mathcal O_{\Pp V^\vee} (1)) - \xi_*(t_1) = 2 s_1^1 - (\beta_1 + \gamma),
    $$
    where we are using the formulas in \cite[Lemma 7.3.]{Eric} to pushforward along $\xi : B \mathbb G_m^2 \to BG$. 
\end{proof}

\begin{lemma}\label{lemma: class XY}
    The class of $\{XY\} \subset \Pp \Sym^2 (V_G^\vee)$ is $s_2^2 +(\gamma -\beta_1)s_2^1 +2\beta_2s_2^0$
\end{lemma}
\begin{proof}
    Let $\pi : \Pp \Sym^2 (V_G^\vee) \to BG$ be the natural map. Note that $XY$ is a section of the bundle $\Sym^2 (V_G^\vee )\otimes \Gamma_G \otimes \det(V_G)$ over $BG$. Define $Q$ as the quotient 
    \begin{equation}\label{eqn: quotient bundle}
        0 \longrightarrow \mathcal O_{BG} \stackrel{XY}{\longrightarrow} \Sym^2 (V_G^\vee) \otimes \Gamma_G \otimes \det(V_G) \longrightarrow Q \longrightarrow 0.
    \end{equation}
    There is a natural diagram on $\Pp \Sym^2 (V_G^\vee)$
    $$
    \begin{tikzcd}
	&& 0 \\
	&& {\mathcal O_{\Pp \Sym^2 (V_G^\vee)}} \\
	0 & {\mathcal O_{\Pp \Sym^2 (V_G^\vee)} (-1)\otimes \Gamma_G \otimes \det(V_G)} & {\pi^* \Sym^2 (V_G^\vee) \otimes \Gamma_G \otimes \det (V_G)} \\
	&& {\pi^*Q} \\
	&& 0
	\arrow[from=3-1, to=3-2]
	\arrow[from=3-2, to=3-3]
	\arrow[from=1-3, to=2-3]
	\arrow["XY", from=2-3, to=3-3]
	\arrow[from=3-3, to=4-3]
	\arrow[from=4-3, to=5-3]
	\arrow["\psi"', from=3-2, to=4-3]
\end{tikzcd}
    $$
    and $\{XY\} \subset \Pp \Sym^2 (V^\vee)$ is precisely the locus where
    $$
    \psi \in  H^0(\Pp\Sym^2 (V_G^\vee), \pi^* Q \otimes \mathcal{O}_{\Pp\Sym^2 (V_G^\vee)}(1) \otimes \Gamma_G^\vee \otimes \det(V_G^\vee))
    $$
     vanishes. Since the codimension of $\{XY\} \subseteq \Pp \Sym^2 (V_G^\vee)$ is $2$, the expected one, we have by \cite[Theorem 14.4]{Fu} that
    $$
    [\{XY\}] = c_2 (\pi^* Q \otimes \mathcal{O}_{\Pp\Sym^2 (V_G^\vee)}(1) \otimes \Gamma_G^\vee \otimes \det(V_G^\vee)).
    $$
    If $r_1, r_2$ are the Chern roots of $V_G^\vee$ (so $r_1 + r_2 =-\beta_1$ and $r_1r_2 = \beta_2$) then, using the relations in $\CH^*(BG)$, we see that the total Chern class of $\Sym^2 (V_G^\vee) \otimes \Gamma \otimes \det(V_G)$ is
    \begin{align*}
        c(\Sym^2 (V_G^\vee) \otimes \Gamma_G \otimes \det(V_G))=&(1+2r_1 + \gamma + \beta_1)(1 + r_1+r_2 +\gamma +\beta_1)(1+2r_2 +\gamma + \beta_1)\\
        =&(1+\gamma)(1+\gamma +r_1-r_2)(1+\gamma +r_2-r_1)\\
        =&(1+\gamma)((1+\gamma)^2-(r_1-r_2)^2)\\
        =&(1+\gamma)(1+2\gamma +\gamma^2 -\beta_1^2 + 4\beta_2)\\
        =&(1+\gamma)(1-\beta_1(\gamma+\beta_1) + 4\beta_2)\\
        =&1 +\gamma -\beta_1(\gamma+\beta_1) + 4\beta_2,
    \end{align*}
     which is the same as the total Chern class of $Q$ by the sequence in \eqref{eqn: quotient bundle}. Therefore,
    \begin{align*}
        [\{XY\}]&=c_2 (\pi^* Q \otimes \mathcal{O}_{\Pp\Sym^2 (V_G^\vee)}(1) \otimes \Gamma_G^\vee \otimes \det(V_G^\vee))\\
        &=c_2(\pi^*Q) + c_1(\pi^*Q)(c_1(\mathcal O_{\Pp \Sym^2 (V_G^\vee)} (1)) -\gamma -\beta_1) + (c_1(\mathcal O_{\Pp \Sym^2 (V_G^\vee)} (1)) -\gamma -\beta_1)^2\\
        &=4\beta_2 -\beta_1(\gamma + \beta_1) + \gamma s_2^1 -\gamma(\gamma +\beta_1) + s^2_2 +\beta_1 s^1_2 - 2\beta_2 -2(\gamma +\beta_1) s_2^1+ (\gamma +\beta_1)^2\\
        &=s_2^2 +(\gamma -\beta_1)s_2^1 +2\beta_2,
    \end{align*}
    where we have used the identity \eqref{eqn: h^2} to write $c_1(\mathcal O_{\Pp \Sym^2 (V_G^\vee)} (1))^2$ in terms of the $s_2^i$.
\end{proof}

\begin{lemma}\label{lem: finite subsets}The following relations hold in $\CH_0(\Sym^k (V_G^\vee))$ for various $k$:
    \begin{itemize}
        \item $[\{ X^2 , Y^2\}] = 2s_2^2 -2\beta_1s_2^1+2(\beta_1^2-2\beta_2)s_2^0$.
        \item $[\{X^2Y, XY^2\}] = 2s_3^3-(3\beta_1 +\gamma)s_3^2 + 2(2\beta_2+\beta_1^2)s_3^1-6\beta_2\beta_1s_3^0$.
        \item $[\{X^3, Y^3\}] = 2s_3^3 +(\gamma-3\beta_1)s_3^2 +2(3\beta_1^2-6\beta_2)s_3^1  + 6\beta_1(3\beta_2-\beta_1^2)s_3^0$.
        \item $[\{X^2Y^2\}]= s_4^4 -2\beta_1s_4^3 +2(\beta_1^2+2\beta_2)s_4^2 -12\beta_2\beta_1s_4^1 + 24\beta_2^2s_4^0$.
        \item $[\{X^3Y, XY^3\}] =2s_4^4-4\beta_1s_4^3+6\beta_1^2s_4^2-6\beta_1^3s_4^1 +24\beta_2(\beta_1^2-2\beta_2)$.
        \item $[\{X^4, Y^4\}] = 2s_4^4-4\beta_1s_4^3+12(\beta_1^2-2\beta_2)s_4^2 + 24\beta_1(\beta_2-\beta_1^2)s_4^1 + 24(\beta_1^4 + 2\beta_2^2-4\beta_2\beta_1^2)s_4^0$.
    \end{itemize}
\end{lemma}
\begin{proof}
    In this proof we will repeatedly apply formula \eqref{eqn: multiplication maps} to compute the pushforwards along $\mathrm{mult}$. In $\CH_0(\Pp \Sym^2 (V_G^\vee))$ we have 
    \begin{align*}
        [\{X^2,Y^2\}] &= \mathrm{mult}_* ([\{X,Y\}],[\{X,Y\}]) -2 [\{XY\}]\\
        &= \mathrm{mult}_*(2s_1^1 - (\beta_1 + \gamma)s_1^0,2s_1^1 - (\beta_1 + \gamma)s_1^0) - 2(s_2^2+(\gamma -\beta_1)s_2^1+2\beta_2s_2^0)\\
        &=  2 s_2^2 + (-4(\beta_1+\gamma) -2(\gamma-\beta_1)) s_2^1+\left(\binom{2}{1}(\beta_1+\gamma)^2 -4\beta_2\right)s_2^0\\
        &=2s_2^2 -2\beta_1s_2^1+2(\beta_1^2-2\beta_2)s_2^0.
    \end{align*}
    In $\CH_0(\Pp \Sym^3 (V_G^\vee))$ we have
    \begin{align*}
        [\{X^2Y, XY^2\}] &= \mathrm{mult}_*([\{X,Y\}], [\{XY\}])\\
        &=\mathrm{mult}_*(2s_1^1-(\beta_1 +\gamma)s_1^0, s_2^2 +(\gamma-\beta_1)s_2^1 + 2\beta_2s_2^0)\\
        &=s_3^3+(2(\gamma-\beta_1)-(\beta_1 +\gamma))s_3^2 + \left(4\beta_2-\binom{2}{1}(\beta_1 + \gamma)(\gamma-\beta_1)\right)s_3^1 - 2\binom{3}{1}(\beta_1 + \gamma)\beta_2s_3^0\\
        &=2s_3^3-(3\beta_1 +\gamma)s_3^2 + 2(2\beta_2+\beta_1^2)s_3^1-6\beta_2\beta_1s_3^0
    \end{align*}
      and
    \begin{align*}
        [\{X^3, Y^3\}] &= \mathrm{mult}_*([\{X,Y\}], [\{X^2, Y^2\}]) - [\{X^2Y, XY^2\}]\\
        &=\mathrm{mult}_*(2s_1^1 - (\beta_1 + \gamma)s_1^0, 2s_2^2 -2\beta_1s_2^1+2(\beta_1^2-2\beta_2)s_2^0)\\
        &- \left(2s_3^3-(3\beta_1 +\gamma)s_3^2 + 2(2\beta_2+\beta_1^2)s_3^1-6\beta_2\beta_1 s_3^0\right)\\
        &= 4s_3^3-(2(\beta_1 + \gamma) +4\beta_1)s_3^2 + \left(2\beta_1\binom{2}{1}(\beta_1 +\gamma)+4(\beta_1^2-2\beta_2)\right)s_3^1 \\ 
        &- \binom{3}{1}2(\beta_1 + \gamma)(\beta_1^2-2\beta_2)s_3^0-2s_3^3+(3\beta_1 +\gamma)s_3^2 - 2(2\beta_2+\beta_1^2)s_3^1+6\beta_2\beta_1s_3^0\\
        &=2s_3^3 +(\gamma-3\beta_1)s_3^2 +6(\beta_1^2-2\beta_2)s_3^1  + 6\beta_1(3\beta_2-\beta_1^2)s_3^0.
    \end{align*}
    In $\CH_0(\Pp \Sym^4 (V_G^\vee))$ we have
    \begin{align*}
        [\{X^2Y^2\}] &= \mathrm{mult}_*([\{XY\}, \{XY\}])\\
        &=\mathrm{mult}_*(s_2^2 +(\gamma -\beta_1)s_2^1 +2\beta_2s_2^0,s_2^2 +(\gamma -\beta_1)s_2^1 +2\beta_2s_2^0)\\
        &=s_4^4 + 2(\gamma-\beta_1)s_4^3 + \left(\binom{2}{1}(\gamma-\beta_1)^2+4\beta_2\right)s_4^2 + 4\binom{3}{1}\beta_2(\gamma-\beta_1)s_4^1 + 4\binom{4}{2}\beta_2^2s_4^0\\
        &=s_4^4 -2\beta_1s_4^3 +2(\beta_1^2+2\beta_2)s_4^2 -12\beta_2\beta_1s_4^1 + 24\beta_2^2s_4^0
    \end{align*}
    and
    \begin{align*}
        [\{X^3Y, XY^3\}] &= \mathrm{mult}_*([\{XY\}, \{X^2,Y^2\}])\\
        &=\mathrm{mult}_*(s_2^2 +(\gamma -\beta_1)s_2^1 +2\beta_2s_2^0, 2s_2^2 -2\beta_1s_2^1+2(\beta_1^2-2\beta_2)s_2^0)\\
        &=2s_4^4 + (2(\gamma-\beta_1)-2\beta_1)s_4^3 + \left(4\beta_2+2(\beta_1^2-2\beta_2) -\binom{2}{1}2\beta_1(\gamma-\beta_1)\right)s_4^2\\
        &+\binom{3}{1}(-4\beta_2\beta_1+2(\gamma-\beta_1)(\beta_1^2-2\beta_2))s_4^1 + \binom{4}{2}4(\beta_1^2-2\beta_2)\beta_2s_4^0\\
        &=2s_4^4-4\beta_1s_4^3+6\beta_1^2s_4^2-6\beta_1^3s_4^1 +24\beta_2(\beta_1^2-2\beta_2) s_4^0,
    \end{align*}
    and 
    \begin{align*}
        [\{X^4, Y^4\}] &= \mathrm{mult}_*([\{X^2,Y^2\}, \{X^2,Y^2\}]) - 2[\{X^2Y^2\}]\\
        &=\mathrm{mult}_*(2s_2^2 -2\beta_1s_2^1+2(\beta_1^2-2\beta_2)s_2^0, 2s_2^2 -2\beta_1s_2^1+2(\beta_1^2-2\beta_2)s_2^0)\\
        &-2(s_4^4 -2\beta_1s_4^3 +2(\beta_1^2+2\beta_2)s_4^2 -12\beta_2\beta_1s_4^1 + 24\beta_2^2s_4^0)\\
        &=   4s_4^4 -8\beta_1s_4^3 +\left(\binom{2}{1}4\beta_1^2+8(\beta_1^2-2\beta_2)\right)s_4^2 - 8\binom{3}{1}\beta_1(\beta_1^2-2\beta_2)s_4^1 +4 \binom{4}{2}(\beta_1^2-2\beta_2)^2s_4^0   \\
        &-2s_4^4 +4\beta_1s_4^3 -4(\beta_1^2+2\beta_2)s_4^2 +24\beta_2\beta_1s_4^1 -48\beta_2^2s_4^0\\
        &=2s_4^4-4\beta_1s_4^3+12(\beta_1^2-2\beta_2)s_4^2 + 24\beta_1(\beta_2-\beta_1^2)s_4^1 + 24(\beta_1^4 + 2\beta_2^2-4\beta_2\beta_1^2)s_4^0.
    \end{align*}
\end{proof}

\subsection{Proof of Theorem \ref{thm: Chow of R2}}\label{sec: computation of Chow}
Recall that, by the projective bundle formula,
$$
\CH^*(\Pp \Sym^4(V_G^\vee)) = \frac{\CH^*(BG)))[h]}{(P(h))} = \frac{\Z[\beta_1, \beta_2, \gamma, h]}{(2 \gamma , \gamma (\gamma + \beta_1), P(h))},
$$
where
$$
P(h) = h^5 + h^4 c_1(\Sym^4(V_G^\vee)) + \ldots + c_5(\Sym^4(V_G^\vee)).
$$
Therefore, by Lemma \ref{lem: chow of Gm torsor}, and Proposition \ref{prop: p is envelope}
\begin{equation}\label{eqn: quotient ring 1}
    \CH^*(\mathrm{Sym}^4 (V_G^\vee) \otimes \mathrm{det}(V_G) \otimes \Gamma \smallsetminus \Delta) = \frac{Z[\beta_1, \beta_2, \gamma]}{(2\gamma, \gamma(\gamma+\beta_1), P(h), \mathrm{im}(p'_*))},
\end{equation}
where
$$
p' : \underline{\Delta}' \longrightarrow \Pp(\Sym^4(V^\vee_G))
$$
is the composition of $p$ and the inclusion  $\underline{\Delta} \subseteq \Pp \Sym^4(V^\vee_G)$ and we are substituting $h=\beta_1 + \gamma$ in $P$ and in $\mathrm{im(p'_*)}$. Let $I \subset \Z[\beta_1, \beta_2, \gamma]$ be the ideal in the denominator of \eqref{eqn: quotient ring 1}. We will show that 
\begin{equation}\label{eqn: ideal I}
I = (2 \gamma, 2\beta_1, 8\beta_2, \gamma^2+ \beta_1 \gamma , \beta_1 ^2+\beta_1\gamma)
\end{equation}
and this will prove Theorem \ref{thm: Chow of R2} (up to identification of the classes $\beta_i$ and $\gamma$ done in \S\ref{sec: tautological classes}).

\begin{lemma}
    $\{2\beta_1, \beta_1^2+\beta_1\gamma, 8\beta_2\}$ are in $I$.
\end{lemma}
\begin{proof}
    From Lemmas \ref{lemma: class XY} and \ref{lem: finite subsets}, we have
    \begin{align*}
    (p_{1,2})_*(s_3^0 \otimes [\{X,Y\}]) &= 2 \mathrm{mult}(s_3^0 , s_1^1) - (\beta_1 + \gamma)\mathrm{mult}(s_3^0, s_1^0)=2s_4^1 -4\beta_1\\
    (p_{1,2})_*(s_3^1 \otimes [\{X,Y\}]) &= 2 \mathrm{mult}(s_3^1 , s_1^1) - (\beta_1 + \gamma)\mathrm{mult}(s_3^1, s_1^0) = 2s_4^2 - 3(\beta_1 +\gamma) s_4^1\\
    (p_{1,1})_*(s_1^1 \otimes s_2^0) &= \mathrm{mult} (\mathrm{sq}(s_1^1), s_2^0) = \mathrm{mult}(s_2^2, s_2^0) - 2\beta_2\mathrm{mult}(s_2^0, s_2^0)\\
    &= s_4^2 -12\beta_2.
    \end{align*}
    Which, after using \eqref{eqn: h^2}, the substitution $h = \beta_1 + \gamma$, and the relations coming from $\CH^*(BG)$, yield $-2\beta_1$, $ 8\beta_2 -3\beta_1(\beta_1 + \gamma)$ and $-8\beta_2$, respectively, and linear combinations of these give the desired elements.
\end{proof}
From now on, and until the rest of the section, we work modulo the relations
\begin{equation}\label{eqn: relations}
2\beta_1= 2\gamma=8\beta_2=0, \beta_1^2 =\beta_1\gamma =\gamma^2,
\end{equation}
and without mentioning that we are substituting $h=\beta_1 +\gamma$. Note that $h^2=0$ modulo the relations \eqref{eqn: relations}.

Also, note that, by \eqref{eqn: h^2}, \eqref{eqn: h^3} and \eqref{eqn: h^4}
\begin{align*}
s_4^2 &= h^2-\beta_1h + 4\beta_2=4\beta_2,\\
s_4^3 &= h^3 -3\beta_1 s_4^2 -(\beta_1^2-10\beta_2)s_4^1 + 4\beta_1\beta_2 =0,\\
s_4^4 &= h^4 - 6\beta_1 s_4^3 - (7\beta_1^2 - 16\beta_2)s_4^2 + (-\beta_1^3 +32\beta_1\beta_2)s_4^1 + 4\beta_1^2\beta_2-40\beta_2^2=0.
\end{align*}
An immediate but useful consequence is the following.

\begin{remark}\label{rmk: vanishing}
    We have $\alpha\cdot\mathrm{mult}(s_a^{a'}, s_{4-a}^{b'})=0$ whenever $(\beta_1+\gamma) \mid \alpha$, or $a'+b' \geq 1$ and $2 \mid \alpha$, or $a'+b' \geq 3$.
\end{remark}

\begin{lemma}
    Modulo the relations \eqref{eqn: relations}, we have 
    $$
    \mathrm{im} (p_*')=0.
    $$ 
\end{lemma}
\begin{proof}
    We have to pushforward all the $s_r^i$ classes along maps $p_{i,j}$. Note that $p_{2,2}$ and $p_{3,2}$ factor through $p_{1,1}$, and so there is no need to consider them. Remark \ref{rmk: vanishing} will be used repeatedly and without further mention.\\
    \underline{Pushforward along $p_{1,1} = \mathrm{mult}(\mathrm{sq}(\cdot), \cdot)$}:
    \begin{align*}
    (p_{1,1})_*(s_1^0\otimes s_2^0) &=2\mathrm{mult}(s_2^1,s_2^0) -2\beta_1\mathrm{mult}(s_2^0, s_2^0)=0\\
    (p_{1,1})_*(s_1^0\otimes s_2^1) &=2\mathrm{mult}(s_2^1,s_2^1) -2\beta_1\mathrm{mult}(s_2^0, s_2^1)=0\\
    (p_{1,1})_*(s_1^0\otimes s_2^2) &=2\mathrm{mult}(s_2^1,s_2^1) -2\beta_1\mathrm{mult}(s_2^0, s_2^1)=0\\
    (p_{1,1})_*(s_1^1\otimes s_2^0) &=\mathrm{mult}(s_2^2,s_2^0) -2\beta_2\mathrm{mult}(s_2^0, s_2^0)=s_4^2 -12\beta_2=0\\
    (p_{1,1})_*(s_1^1\otimes s_2^1) &=\mathrm{mult}(s_2^2,s_2^1) -2\beta_2\mathrm{mult}(s_2^0, s_2^1)=0\\
    (p_{1,1})_*(s_1^1\otimes s_2^2) &= \mathrm{mult}(s_2^2,s_2^2) -2\beta_2\mathrm{mult}(s_2^0, s_2^2)=0
    \end{align*}
    \underline{Pushforward along $p_{1,2} = \mathrm{mult}(\cdot, \{X,Y\})$}:
    \begin{align*}
        (p_{1,2})_*(s_3^0) &= 2\mathrm{mult}(s_3^0 , s_1^1)-(\beta_1 +\gamma)\mathrm{mult}(s_3^0, s_1^0)=0 \\
        (p_{1,2})_*(s_3^1) &= 2\mathrm{mult}(s_3^1 , s_1^1)-(\beta_1 +\gamma)\mathrm{mult}(s_3^1, s_1^0) =0 \\
        (p_{1,2})_*(s_3^2) &= 2\mathrm{mult}(s_3^2 , s_1^1)-(\beta_1 +\gamma)\mathrm{mult}(s_3^2, s_1^0) =0 \\
        (p_{1,2})_*(s_3^3) &= 2\mathrm{mult}(s_3^3 , s_1^1)-(\beta_1 +\gamma)\mathrm{mult}(s_3^3, s_1^0) =0
    \end{align*}
    \underline{Pushforward along $p_{2,1} = \mathrm{sq}(\cdot)$}:
    \begin{align*}
        (p_{2,1})_*(s_2^0) &=  4s_4^2 - 12\beta_1s_4^1 + 12\beta_1^2=0\\
        (p_{2,1})_*(s_2^1) &=  2s_4^3 -2\beta_1 s_4^2 -12\beta_2 s_4^1+24 \beta_1\beta_2=0\\
        (p_{2,1})_*(s_2^2) &= s_4^4 -4\beta_2 s_4^2+24\beta_2^2=0,
    \end{align*}
    see Lemma \ref{lem: the diagonal of Psym}. \\
    \underline{Pushforward along $p_{2,3} = \mathrm{mult}(\cdot , \{X^2, Y^2, XY\})$}
    \begin{align*}
        (p_{2,3})_*(s_2^0\otimes [\{X^2, Y^2\}])&=2\mathrm{mult}(s_2^0,s_2^2) -2\beta_1\mathrm{mult}(s_2^0,s_2^1)+2(\beta_1^2-2\beta_2)\mathrm{mult}(s_2^0,s_2^0)=0\\
        (p_{2,3})_*(s_2^1\otimes [\{X^2, Y^2\}])&=2\mathrm{mult}(s_2^1,s_2^2) -2\beta_1\mathrm{mult}(s_2^1,s_2^1)+2(\beta_1^2-2\beta_2)\mathrm{mult}(s_2^1,s_2^0)=0\\
        (p_{2,3})_*(s_2^2\otimes [\{X^2, Y^2\}])&= 2\mathrm{mult}(s_2^2,s_2^2) -2\beta_1\mathrm{mult}(s_2^2,s_2^1)+2(\beta_1^2-2\beta_2)\mathrm{mult}(s_2^2,s_2^0)=0\\
        (p_{2,3})_*(s_2^0\otimes [\{XY\}])&= \mathrm{mult}(s_2^0,s_2^2) +(\gamma -\beta_1)\mathrm{mult}(s_2^0,s_2^1) +2\beta_2\mathrm{mult}(s_2^0,s_2^0)\\
        &=s_4^2+12\beta_2=0\\
        (p_{2,3})_*(s_2^1\otimes [\{XY\}])&= \mathrm{mult}(s_2^1,s_2^2) +(\gamma -\beta_1)\mathrm{mult}(s_2^1,s_2^1) +2\beta_2\mathrm{mult}(s_2^1,s_2^0)=0\\
        (p_{2,3})_*(s_2^2\otimes [\{XY\}])&= \mathrm{mult}(s_2^2,s_2^2) +(\gamma -\beta_1)\mathrm{mult}(s_2^2,s_2^1) +2\beta_2\mathrm{mult}(s_2^2,s_2^0)=0
    \end{align*}
    \underline{Pushforward along $p_{3,1} = \mathrm{mult}(\cdot, \{X^3, Y^3, X^2Y, XY^2\})$}:
    \begin{align*}
        (p_{3,1})_*(s_1^0, [\{X^3, Y^3\}])  =\,& 2\mathrm{mult}(s_1^0,s_3^3) +(\gamma-3\beta_1)\mathrm{mult}(s_1^0,s_3^2)\\
        &+2(3\beta_1^2-6\beta_2)\mathrm{mult}(s_1^0,s_3^1)  + 6\beta_1(3\beta_2-\beta_1^2)\mathrm{mult}(s_1^0,s_3^0)=0\\
        (p_{3,1})_*(s_1^1, [\{X^3, Y^3\}]) =\,& 2\mathrm{mult}(s_1^1,s_3^3) +(\gamma-3\beta_1)\mathrm{mult}(s_1^1,s_3^2)\\
        &+2(3\beta_1^2-6\beta_2)\mathrm{mult}(s_1^1,s_3^1)  + 6\beta_1(3\beta_2-\beta_1^2)\mathrm{mult}(s_1^1,s_3^0)=0 \\
        (p_{3,1})_*(s_1^0, [\{X^2Y, XY^2\}])  =\,& 2\mathrm{mult}(s_1^0,s_3^3)-(3\beta_1 +\gamma)\mathrm{mult}(s_1^0,s_3^2)\\
        &+ 2(2\beta_2+\beta_1^2)\mathrm{mult}(s_1^0,s_3^1)-6\beta_2\beta_1\mathrm{mult}(s_1^0,s_3^0)=0\\
        (p_{3,1})_*(s_1^1, [\{X^2Y, XY^2\}])  =\,& 2\mathrm{mult}(s_1^1,s_3^3)-(3\beta_1 +\gamma)\mathrm{mult}(s_1^1,s_3^2)\\
        &+ 2(2\beta_2+\beta_1^2)\mathrm{mult}(s_1^1,s_3^1)-6\beta_2\beta_1\mathrm{mult}(s_1^1,s_3^0)=0\\
    \end{align*}
    \underline{Pushforward along $p_{4}$}: This is just the fundamental classes of $\{X^4, Y^4\}$, $\{X^3Y, XY^3\}$ and $\{X^2Y^2\}$, which we calculated in Lemma \ref{lem: finite subsets}, and are easily seen to be $0$.
\end{proof}

\begin{lemma}
    Modulo the relations \eqref{eqn: relations}, we have 
    $$
    P(h)=0.
    $$
\end{lemma}
\begin{proof}
    If $r_1, r_2$ are the Chern roots of $V_G$  ($r_1 +r_2=\beta_1$, $r_1r_2 =\beta_2$) then
    $$
    c_5(\Sym^4 (V_G^\vee)) = -\prod_{i=0}^4 (ir_1 +(4-i)r_2)\text{ and } c_4(\Sym^4 (V_G^\vee)) = \sum_{j=0}^4 \prod_{i \neq j}(ir_1 +(4-i)r_2).
    $$
    From this it follows that $2\beta_1 \mid c_5(\Sym^4 (V_G^\vee))$ and $2h \mid h c_4(\Sym^4(V_G^\vee))$. Also, $h^2=0$, so $P(h)=0$.
\end{proof}

\section{Interpretation of the generators}\label{sec: tautological classes}

In this section we will give a geometric interpretation of the generators $\beta_1,\beta_2$ and $\gamma$.

We start with identifying the $\beta$'s with the Chern classes of the Hodge bundle. This is a natural vector bundle $\mathbb{E}$ of rank 2 on $\R_2$: it is pulled back from $\mathcal{M}_2$ and if $\pi: C \to S$ is a family of curves of genus $2$ corresponding to a morphism $S \to \mathcal{M}_2$, and $\omega_\pi$ is the relative dualizing sheaf of $\pi$, then the pullback of $\mathbb{E}$ to $S$ is $\pi_*(\omega_\pi)$. The Chern classes $\lambda_i=c_i(\mathbb{E})$ are among the tautological classes introduced by Mumford.

\begin{proposition}\label{prop: Hodge bundle}
    The pullback of $V_G$ to $\mathcal{X}$ yields the dual of the Hodge bundle, i.e. 
    $$
    \alpha^* \mathbb{E}=V_G^\vee
    $$
    on $\mathcal{X}$.
\end{proposition}

\begin{remark}\label{rmk: mistake in Vistoli}
    In the proof of the above proposition we will refer to \cite[Proposition 3.1]{Vistoli}. However, with the notation of that paper, there is an error in the description of the $\mathrm{GL}_2$-action on the space $X$. It is incorrectly stated that this action, induced by the isomorphism of $X$ with the $\mathrm{GL}_2$-equivariant variety $Y$, identifies $X$ as an open subset of $\Sym^6(V^\vee) \otimes \det(V)^{\otimes 2}$. Instead, $X$ is identified with the open subset of $\Sym^6(V) \otimes \det(V^\vee)^{\otimes 2}$ consisting of polynomials with no repeated roots.
    
    Base changing over the morphism 
    $$
    B\mathrm{GL}_2 \to B\mathrm{GL}_2
    $$
    given by $A \mapsto (A^{-1})^t$, we get an isomorphism
    \begin{equation}\label{eqn: new vistoli iso}
    \mathcal M_2 \cong \left[X'/\operatorname{GL}_2\right],
    \end{equation}
    where $X' \subset \Sym^6(V^\vee) \otimes \det (V)^{\otimes 2}$ is the locus with no repeated roots, and under this new isomorphism, the Hodge bundle is the pullback of $V_{\operatorname{GL}_2}^\vee$ from $B\operatorname{GL}_2$.
\end{remark}

\begin{proof}[Proof of Proposition \ref{prop: Hodge bundle}]
The morphism
\begin{alignat*}{2}
    \chi : \operatorname{Sym}^4(V^\vee) \otimes \det(V) \otimes \Gamma & \longrightarrow && \, \Sym^6(V^\vee) \otimes \det (V)^{\otimes 2}\\
    F & \mapsto &&XYF
\end{alignat*}
is equivariant with respect to the inclusion $\iota : G \to \operatorname{GL}_2$, and so we have a commutative diagram
$$
\begin{tikzcd}
    \mathcal X \arrow[d, "\alpha"] \arrow[r, "\tilde{\chi}"] & \left[X'/\operatorname{GL}_2\right] \arrow[d, "\eqref{eqn: new vistoli iso}"]\\
    \mathcal R_2 \arrow[r] & \mathcal M_2,
\end{tikzcd}
$$
where the bottom arrow is the forgetful morphism. Therefore, the proposition follows because the pullback of the Hodge bundle under \eqref{eqn: new vistoli iso} is $V_{\operatorname{GL}_2}^\vee$, and $\iota^* V_{\operatorname{GL_2}} = V_G$.
\end{proof}

We obtain the following immediate corollary:

\begin{corollary}
    We have 
    $$
    \beta_i=(-1)^i \lambda_i
    $$
    for $i=1,2$.
\end{corollary}

Finally we identify the class $\gamma \in \CH^*(\R_2)$. There is a natural double cover of $\R_2$ which we now describe. Denote by $U$ the affine scheme $\Sym^4(V^\vee) \otimes \mathrm{det}(V) \otimes \Gamma \smallsetminus \Delta$. Over $U$ we have the following commutative diagram of $G$ equivariant maps

\begin{equation*}
    \adjustbox{scale=0.95,center}{
        \begin{tikzcd}
         D=\sigma_0(U) \sqcup \sigma_{\infty}(U) \arrow[hookrightarrow]{r} \arrow{dr}{\delta} & C =\underline{\mathrm{Spec}}_{\P_U}( \mathcal{O}_{\P_U} \oplus \mathcal{O}_{\P_U}(-3))  \arrow{r}{q} \arrow{d}{\pi}  & \P_U =U \times \mathbb{P}(V)  \arrow{dl}{\rho}  \\
        & U  & &
        \end{tikzcd}
    }
    \end{equation*}
The action of $G$ on $C$ is given by Proposition \ref{prop: identification of G} and we see that $D$ is $G$-equivariant. Then $\delta$ descends to 
$$
\bigg[ \frac{D}{G} \bigg] \to \R_2,
$$
and we have the following lemma.

\begin{lemma}\label{lemma: interpretation gamma}
    We have 
    $c_1^G(\delta_* \mathcal{O}_D)= \gamma$.
\end{lemma}
\begin{proof}
    The pushforward of the structure sheaf $\mathcal{O}_D$ under $\delta$ is clearly $\mathcal{O}_U \otimes A$ where $A$ is the two dimensional representation of $G$ arising from the representation of $\mathbb{Z}/(2)$ where $-1$ switches the two entries of the vector. Clearly, $\mathrm{det}(A)=\Gamma$ and the lemma follows.
\end{proof}

$\,$\
\noindent

$\,$\
\noindent
\textsc{Department of Pure Mathematics {\it \&} Mathematical Statistics, 
University of Cambridge, Cambridge, UK}

\textit{e-mail address:} \href{mailto:ac2758@cam.ac.uk}{ac2758@cam.ac.uk}

$\,$\\
\textsc{Department of Mathematics, ETH Zürich, Rämistrasse 101, 8092 Zürich, Switzerland}

\textit{e-mail address:} \href{mailto:aitor.iribarlopez@math.ethz.ch}{aitor.iribarlopez@math.ethz.ch}

\end{document}